\documentclass[10pt]{article}
\usepackage{amsmath,amsfonts,amsthm,amssymb}

\usepackage[utf8]{inputenc}

\usepackage[
    backend=biber,
    style=alphabetic,
    natbib=true,
    url=true, 
    doi=false,
    eprint=true,
    defernumbers
]{biblatex}

\usepackage{latexsym}
\usepackage{mathrsfs}
\usepackage{enumerate}
\usepackage{hyperref,color}
\usepackage{tikz-cd}
\usepackage{rotating}
\newcommand*{\isoarrow}[1]{\arrow[#1,"\rotatebox{90}{\(\sim\)}"
]}

\usepackage[active]{srcltx}
\usepackage{xspace}

\theoremstyle{plain}
\newtheorem{theorem}{Theorem}[section]
\newtheorem{proposition}[theorem]{Proposition}
\newtheorem{corollary}[theorem]{Corollary}
\newtheorem{lemma}[theorem]{Lemma}

\newtheorem{theoremA}{Theorem}

\newtheorem{corollaryA}[theoremA]{Corollary}

\theoremstyle{definition}
\newtheorem{definition}[theorem]{Definition}

\theoremstyle{remark}
\newtheorem{remark}[theorem]{Remark}

\DeclareMathOperator{\Div}{Div}
\DeclareMathOperator{\Spec}{Spec} 

\DeclareMathOperator{\pr}{pr} 
\DeclareMathOperator{\princ}{div} 
\DeclareMathOperator{\Et}{\acute{E}t}

\DeclareMathOperator{\Vect}{Vect} 
\DeclareMathOperator{\Par}{Par} 
\DeclareMathOperator{\Con}{Con}
\DeclareMathOperator{\Sym}{Sym}
\DeclareMathOperator{\SPEC}{\mathbf{Spec}} 
\DeclareMathOperator{\res}{res} 
\DeclareMathOperator{\HOM}{\mathbf{Hom}} 
\DeclareMathOperator{\id}{id} 
\DeclareMathOperator{\Hom}{Hom} 
 
\DeclareMathOperator{\Sesh}{Sesh} 
\DeclareMathOperator{\rk}{rk} 
\DeclareMathOperator{\coker}{coker} 
\DeclareMathOperator{\im}{im} 
\DeclareMathOperator{\Int}{Int} 
\DeclareMathOperator{\Rep}{Rep} 
\DeclareMathOperator{\Sec}{Sec}

\newcommand{\preparabolic}{parabolic~}
\newcommand{\parabolic}{strongly parabolic\xspace}

\begin{document}

\title{Parabolic connections and stack of roots}
\author{Niels Borne \and  Amine Laaroussi}
\maketitle
%\tableofcontents

\section{Introduction}%
\label{sec:introduction}

\subsection{The notion of parabolic connection}%
\label{sub:the_notion_of_parabolic_connection}

A parabolic connection on an algebraic variety $X$ endowed with a divisor $D$ is, roughly, a vector bundle on $X$ equipped with two compatible structures: a parabolic structure in the sense of Mehta-Seshadri, and a logarithmic connection. Parabolic connections and parabolic Higgs bundles have been introduced by Carlos Simpson in order to establish a version of what is now called Simpson's correspondence over a non-compact curve (\cite{Simpson:noncompact}). Simpson's interpretation of parabolic bundles as filtered sheaves led to the generalization of the definition of parabolic Higgs bundles to higher dimensional varieties (\cite{yoko:compactification}).

Since then, parabolic connections, and their moduli spaces, have been an active subject of research, mainly over a curve\footnote{We apologize not to be able to cite the numerous contributions to this nice subject.}. However, parabolic connections also made quite recently a notable apparition on higher varieties in the work of R.Donagi and T.Pantev on Geometric Langlands Conjecture using Simpson's non abelian Hodge theory (\cite{donagi2019parabolic}).

\subsection{First stacky interpretations}%
\label{sub:first_stacky_interpretations}

Meanwhile, another interpretation of parabolic bundles as orbifold bundles came up, first on global quotients (\cite{bis:parabolic-bundles-orbifold-bundles}) then on natural algebraic stacks associated to $(X,D)$, the stacks of roots (\cite{bor:corr,bor:rep,BV:par}). More precisely, there is a Fourier like correspondence between parabolic bundles and ordinary vector bundles on the stack of roots.

This raises the question of understanding parabolic connections through this correspondence. This question was answered in dimension $1$ by Biswas-Majumder-Wong (\cite{BMW:root}) and Loray-Saito-Simpson (\cite[]{LSS:foliations}). Both teams came to the same conclusion: connections on the stack of roots that are holomorphic correspond precisely to parabolic connections such that the weights of the parabolic structure are the spectra of the residues of the connection. Also notable was Biswas-Majumder-Wong's similar description of holomorphic Higgs bundles on the stack roots on a variety of arbitrary dimension (\cite{BMW:higgs}). 
\subsection{Our results}%
\label{sub:our_results}

Our main goal is to generalize the results above in any dimension. Our starting data is a smooth variety $X$ over a field $k$ endowed with a smooth divisor $D$\footnote{In the main text, we work more generally with a strict normal crossings divisor.}. Let $r\in \mathbb N^*$.  To this data, we can associate on one hand the corresponding stack of roots $\pi:\mathfrak X \to X$, this is the minimal stack such that $\pi^* \left(\frac{1}{r}D\right)$ is integral. On the other hand, we can define \preparabolic connections, that is decreasing families $(\mathcal E_\alpha,\nabla_\alpha)_{\alpha\in \frac{1}{r}\mathbb Z }$ of logarithmic connections such that $\mathcal E_{\alpha+1}\simeq \mathcal E_\alpha(-D)$. Consistently with the current terminology on parabolic Higgs bundles, we say that a logarithmic connection is \emph{\parabolic} if moreover the residue of $\nabla_0$ is semi-simple with eigenvalues the weights of the underlying parabolic bundle. Our first main result is:

\begin{theoremA}[Theorem \ref{thm:car_hol_stacky}]
\label{thm:car_hol_stackyA}

A logarithmic connection $(\mathcal F,\nabla)$ on $\mathfrak X$ is holomorphic if and only if the corresponding \preparabolic connection $(\mathcal E_\alpha,\nabla_\alpha)_{\alpha\in \frac{1}{r}\mathbb Z }$ is \parabolic.
\end{theoremA}

From Theorem \ref{thm:car_hol_stackyA} and the usual stacky-parabolic equivalence for vector bundles we deduce:

\begin{theoremA}[Theorem \ref{thm:correspondence}]
\label{thm:correspondenceA}
There is a natural tensor equivalence of categories between holomorphic connections on $\mathfrak{X}$ and \parabolic connections with weights in $\frac{1}{r} \mathbb{Z}$. 
\end{theoremA}

Finally, inspired by \cite{IS:relation}, we show that if $(\mathcal E_\alpha,\nabla_\alpha)_{\alpha\in \frac{1}{r}\mathbb Z }$ is a \parabolic connection, the connection $\nabla_0$ on the underlying bundle $\mathcal{E}_0$ enables to reconstruct the parabolic structure. Via Theorem \ref{thm:correspondenceA}, this has the following rather surprising translation:

\begin{corollaryA}[Corollary \ref{cor:correspondence}]
\label{cor:correspondenceA}
Let $(\mathcal F,\nabla)$ and $(\mathcal F',\nabla')$ be two holomorphic connections on $\mathfrak{X}$, and $\pi:\mathfrak X \to X$ be the morphism to the moduli space. Then any isomorphism 
$(\pi_*\mathcal F,\pi_*\nabla)\simeq (\pi_*\mathcal F',\pi_*\nabla')$ lifts uniquely  to an isomorphism $(\mathcal F,\nabla)\simeq(\mathcal F',\nabla')$.

\end{corollaryA}

The corresponding statements for vector bundles or even for logarithmic connections are easily seen to be false.

\subsection{Content}%
\label{sub:content}

We now give more details about the structure of the article.

The first section (\S \ref{sec:stacky_preliminaries}) is a reminder of well-known results on stacks of roots. We recall how the hypothesis that we consider a strict normal crossings divisor implies that the stack of roots is smooth. We then turn to the definition of parabolic sheaves and their correspondence with vector bundles on the stack of roots.

In the next section (\S \ref{sec:connections_on_dm_stacks}), we concentrate on connections on Deligne-Mumford stacks. Our main reference is Martin Olsson's books (\cite[]{ols:book,ols:cry}). As for vector bundles, the small étale site is sufficient to get a good notion of a connection. We recall how Atiyah's exact sequence enables to see connections within the $\mathcal O$-linear world. Our next task is to define logarithmic connections. Even the definition of logarithmic differentials is a bit tricky, as the usual Zariski local definition on schemes is not canonical enough to be useful when it comes to Deligne-Mumford stacks. Instead, we use Martin Olsson fundamental insight that logarithmic differentials should be seen as (the pull-back of) the sheaf $\Omega^1_{ \mathbb{A}^1/[\mathbb{A}^1/ \mathbb{G}_m]}$. Even for schemes, this gives a global definition of logarithmic differentials that does not seem to be well-known, but that is intrinsic and generalizes immediately to Deligne-Mumford stacks.

In the main section (\S \ref{sec:the_correspondence}), we define parabolic connections, and interpret them as sections of the parabolic Atiyah exact sequence. We then show a reconstruction theorem à la Iyer-Simpson. We finally prove Theorem \ref{thm:car_hol_stackyA} and deduce Theorem \ref{thm:correspondenceA}. Despite the apparent similarity between this last theorem and the previous result for vector bundles (Theorem \ref{thm:description}), the proof is very different\footnote{In fact the proof of Theorem \ref{thm:correspondenceA} relies on Theorem \ref{thm:description}.}. The reason is that the strategy of the proof Theorem \ref{thm:description} does not work for connections, as they are not Zariski-locally sum of objects of rank $1$.

Finally, the last section (\S \ref{sec:towards_the_log-kummer_algebraic_fundamental_group}) contains some thoughts on a potential definition of the log-Kummer algebraic fundamental group.

\subsection{Conventions}%
\label{sub:conventions}

\subsubsection{Base field}%
\label{ssub:base_field}

We fix a base field $k$, often assumed to be perfect, and set $S=\Spec k$. In some cases, we will need to work over an arbitrary base scheme $S$, this will then be mentioned explicitly. 

\subsubsection{Algebraic stacks}%
\label{ssub:algebraic_stacks}

We follow the conventions of \cite{ols:book}: in particular, we consider stacks on the category $Sch/S$ of $S$-schemes endowed with the étale topology (the big étale site of $S$).

\subsubsection{Logarithmic and log-smooth context}%
\label{ssub:logarithmic_context}

In this context, we fix a $k$-scheme $X$ and a finite family $\mathbf D=(D_i)_{i\in I}$ of distinct effective integral Cartier divisors.

Most of the time, we assume that $X$ is a smooth $k$-scheme and that moreover  $D=\cup_{i\in I}D_i$ is a strict normal crossings divisor (\S \ref{ssub:normal_crossings}). We will then say that we are in the log-smooth context.

To this data, we will add a system of weights $\mathbf r= (r_i)_{i\in I}$, where each $r_i$ is a positive integer. This allows to define the stack of roots $\sqrt[\mathbf r]{X/\mathbf D}$, often denoted by $\mathfrak X_{\mathbf r}$, or even by $\mathfrak X$, when there is no ambiguity (\S \ref{ssub:stack_of_roots}). We denote by $\pi_{\mathbf r}$, or more often by $\pi$, the natural morphism $\mathfrak X_{\mathbf r}\rightarrow X$.  Each $D_i$ has a canonical $r_i$-th root $\mathfrak D_i$ on $\mathfrak X_{\mathbf r}$.

As we want to stick to Deligne-Mumford stacks, we will assume that each $r_i$ is invertible in $k$.

\subsubsection{Stacky context}%
\label{ssub:stacky_context}

We will also need to work in a more general situation, where $\mathfrak X/k$ is a smooth Deligne-Mumford stack.
Following the conventions in  \cite{ols:book}, quasi-coherent sheaves on $\mathfrak X$ will be considered as sheaves on the small étale site of $\mathfrak X$. Sometimes, the point of view of sheaves on a groupoid will also be useful. 

Finally, we are also in some cases led to endow $\mathfrak X/k$ with a finite family $\mathfrak D =(\mathfrak D_i)_{i\in I}$ of distinct effective integral Cartier divisors, thereby generalizing the situation in \S \ref{ssub:logarithmic_context}. We will use the natural combination of names: stacky logarithmic context, stacky log-smooth  context ...

\subsection{Acknowledgments}%
\label{sub:acknowledgments}

This project started from discussions of the first author with Mattia Talpo and Angelo Vistoli, who we thank heartily.

\section{Generalities on stacks of roots}%
\label{sec:stacky_preliminaries}

\subsection{Definition, flat presentation, and smoothness}%
\label{sub:algebraic_stacks}

For our purposes, one most useful Artin stack is the stack $\Div_S$ of generalized Cartier divisors: objects over the scheme $T\to S$ are pairs $(\mathcal L, s)$ where $\mathcal L$ is an invertible sheaf on $T$ and $s$ is a global section of $\mathcal L$ \footnote{Here our convention differs slightly from \cite{ols:book}.}. It is well-known that $\Div_S$ is isomorphic to the quotient stack $[\mathbb A^1/\mathbb G_m]$ (\cite[Proposition 10.3.7]{ols:book}). Similarly, given a finite set $I$, the power stack $\Div^I$ classifying families of generalized Cartier divisors indexed by $I$ is isomorphic to $[\mathbb A^I/\mathbb G_m^I]$.

\subsubsection{Stack of roots}%
\label{ssub:stack_of_roots}

\begin{definition}[]
\label{def:stack_of_roots}
Let $(\mathcal L,s)$ be a generalized Cartier divisor on a scheme $X$, and $r\in \mathbb N^*$ a positive integer. The stack of roots $\sqrt[r]{(\mathcal{L},s)/X}$ is the stack classifying $r$-th roots of $(\mathcal{L},s)$, that is generalized Cartier divisors $(\mathcal{M},t)$ endowed with an isomorphism $\mathcal{L}\simeq \mathcal{M}^{\otimes r}$ sending $s$ to $t^{\otimes r}$.
\end{definition}
In other words,  $\sqrt[r]{(\mathcal{L},s)/X}$ is given by the $2$-cartesian diagram:

\begin{center}
\begin{tikzcd}
{ \sqrt[r]{(\mathcal{L},s)/X}} \arrow[r] \arrow[d] & Div_S \arrow[d, "\times r"] \\
X \arrow[r, "{(\mathcal L,s)}"']                   & Div_S                      
\end{tikzcd}
\end{center}

The stack $\sqrt[r]{(\mathcal{L},s)/X}$ is Deligne-Mumford if $r$ is invertible on $S$ (\cite[Theorem 10.3.10]{ols:book}).

The diagram above shows that the construction of the stack of roots makes sense if $X$ is an algebraic stack, in particular, we can iterate it. This leads to the stack of roots associated to a finite family 

\[  \left( (\mathbf {\mathcal L},\mathbf{s}), \mathbf{r}\right) = \left( ({\mathcal L_i},s_i), r_i\right)_{i\in I}\]
where each $({\mathcal L_i},s_i)$ is a generalized Cartier divisor on $X$ and $r_i$ is a positive integer. By definition we set 

\[ \sqrt[\mathbf{r}]{{(\mathbf{\mathcal{L}}},\mathbf{s})/X}= \prod_{X,i\in I} \sqrt[r_i]{(\mathcal{L}_i,s_i)/X}\] 
where the fiber product on the right hand-side is taken over $X$.

We will in fact consider stacks of roots associated to genuine effective Cartier divisors, and will identify such a divisor $D$ with the associated generalized Cartier divisor $(\mathcal O_X(D), s_D)$, where $s_D$ denotes the canonical section. In other words, starting from the logarithmic context (\S \ref{ssub:logarithmic_context}), we put 
$$\sqrt[\mathbf{r}]{\mathbf{D}/X}= \sqrt[\mathbf{r}]{{(\mathcal O_X(\mathbf{D}),\mathbf{s})/X}}\;.$$
If $\pi:\mathfrak X \to X$ is the natural morphism to the moduli space, there is for each $i\in I$ a canonical Cartier divisor $\mathbf{\mathfrak D}_i$  on $\mathfrak X=\sqrt[\mathbf{r}]{\mathbf{D}/X}$ such that $\pi^* D_i= r_i \mathbf{\mathfrak D}_i$.

\subsubsection{Canonical flat presentation}%
\label{ssub:canonical_flat_presentation}

To an invertible sheaf $\mathcal L$ on $X$, we associate as usual the $\mathbb G_m$-torsor $\mathbb V(\mathcal L)\backslash\{0\}=\SPEC_X \Sym^{\pm}(\mathcal L)$. So from the data of $(D_i)_{i\in I}$ we get:
\begin{itemize}
	\item a $\mathbb G_m^I$-torsor $p_D:T_D\rightarrow X$, where $T_D= \prod_{X,i\in I}\mathbb V(\mathcal O_X(D_i))\backslash\{0\}$,
	\item a morphism $a_D:T_D\rightarrow \mathbb A^I$ corresponding to the canonical sections of the $D_i$'s.
\end{itemize}

In stacky terms, $T_D=X\times _{[\mathbb A^I/\mathbb G_m^I]}\mathbb A^I$. For the stack of roots $\mathfrak X$, one defines similarly $T_{\mathfrak{D}}=\mathfrak{X}\times _{[\mathbb A^I/\mathbb G_m^I]}\mathbb A^I$. This is a priori an algebraic space but in fact a scheme as $T_{\mathfrak{D}}=T_D\times_{\mathbb A^I}\mathbb A^I$ (\cite[Remark 4.14.]{BV:par}) . So we have a canonical $\mathbb G_m^I$-torsor $p_{\mathfrak{D}}:T_{\mathfrak{D}}\rightarrow \mathfrak{X}$ that enables to identify $\mathfrak{X}$ with the quotient stack $[T_{\mathfrak{D}}/\mathbb G_m^I]$, a very convenient fact to define logarithmic differentials in this context (see \S \ref{ssub:logarithmic_differentials}).

\subsubsection{Normal crossings}%
\label{ssub:normal_crossings}

For the rest of this section, we use the notations of the logarithmic context (\S \ref{ssub:logarithmic_context}).
We wish to give a condition ensuring the smoothness of the stack of roots $\sqrt[\mathbf{r}]{\mathbf{D}/X}$.

\begin{definition}[{\cite[\href{https://stacks.math.columbia.edu/tag/0CBN}{Tag 0CBN}]{stacks-project}}]
\label{def:ncd}
An effective Cartier divisor $D$ on a locally noetherian scheme $X$ has \emph{strict normal crossings} if for each $x\in D$, the local ring $\mathcal O_{X,x}$ is regular and there exists a regular system of parameters $x_1,\cdots,x_n$ in $\mathfrak m_x$ and an integer $m\in [1,n]$ such that $D$ admits $x_1\cdots x_m$ for equation at $x$. It has \emph{normal crossings} if it has strict normal crossings étale locally on $X$.
\end{definition}

We will use the abbreviation ncd (resp. sncd) for normal crossings divisor (resp. strict normal crossings divisor).

The following proposition is folklore (see for instance \cite[Example 2.5]{kato:log} and \cite[Example 1.2]{MO:Kawamata}) but since we couldn't find a proof in the literature, we provide one.

\begin{proposition}
\label{prop:sncd_etale_coordinates}

Let $k$ be a perfect field, $X$ be a locally algebraic $k$-scheme, and $D$ an be an effective divisor. The following are equivalent:
\begin{enumerate}[(i)]
	\item  $X$ is regular, and $D$ has strict normal crossings (resp. normal crossings),
	\item Zariski (resp. étale) locally on $X$ there exists an étale morphism $X\to \mathbb A^n_k=\Spec k[X_1,\cdots,X_n]$ and  an integer $m\in [0,n]$ such that $D$ is the pullback of the divisor given by $X_1\cdots X_m$ on $\mathbb A^n_k$.
\end{enumerate}
\end{proposition}

\begin{proof}
	The resp. claim follows from the main one. The implication $(ii)\implies(i)$ follows from the facts that "Smooth over a field implies regular" \cite[\href{https://stacks.math.columbia.edu/tag/056S}{Tag 056S}]{stacks-project} and "Pullback of a strict normal crossings divisor by a smooth morphism is a strict normal crossings divisor" 
	\cite[\href{https://stacks.math.columbia.edu/tag/0CBP}{Tag 0CBP}]{stacks-project}. Let us show  the implication $(i)\implies(ii)$. As $k$ is perfect, $X$ is $k$-smooth by \cite[\href{https://stacks.math.columbia.edu/tag/0B8X}{Tag 0B8X}]{stacks-project}. Let $x$ be a closed point of $X$, it is enough to show the result around $x$ by \cite[\href{https://stacks.math.columbia.edu/tag/02IL}{Tag 02IL}]{stacks-project}. If $x\notin D$, the result follows from the existence of étale coordinates for smooth schemes \cite[\href{https://stacks.math.columbia.edu/tag/054L}{Tag 054L}]{stacks-project}. If $x\in D$, let $x_1,\cdots,x_n$ be a regular system of parameters as in Definition \ref{def:ncd}. By 
the proof of \cite[\href{https://stacks.math.columbia.edu/tag/00TV}{Tag 00TV}]{stacks-project}, the sequence

\[
	0\rightarrow \frac{\mathfrak m_x}{\mathfrak m_x^2}\xrightarrow{d}
\left(\Omega^1_{X/k}\right)_x\otimes_{\mathcal O_{X,x}} k(x)  \rightarrow \Omega^1_{k(x)/k}\rightarrow 0
\] 
is exact, and since $x$ is a closed point and $k$ is perfect, we have also that $\Omega^1_{k(x)/k}=0$. Hence $(dx_1,\cdots,dx_n)$ form a basis of $\left(\Omega^1_{X/k}\right)_x$, and by \cite[\S 2.2 Corollary 10]{BLR:Neron}, the morphism $(x_1,\cdots,x_n):X\rightarrow \mathbb A^n$ is étale at $x$ (that is $(x_1,\cdots,x_n)$ are étale coordinates at $x$), and the claim follows.
\end{proof}

We are now able to prove the smoothness of the stack of roots with respect a sncd divisor.

\begin{proposition}
\label{prop:sncd_implies_smooth}
Let $k$ be a perfect field, $X$ be a smooth $k$-scheme, $(D_i,r_i)_{i\in I}$ be a finite family of distinct effective integral Cartier divisors endowed with positive integers, invertible in $k$. Assume that the divisor $D=\cup_{i\in I}D_i$ is a sncd. Then the stack of roots $\mathfrak X=\sqrt[\mathbf{r}]{{(\mathbf{D},\mathbf{s})/X}}$ is $k$-smooth.
\end{proposition}

\begin{proof}
%	Denote by $\mathfrak X=\sqrt[\mathbf{r}]{{(\mathbf{D},\mathbf{s})/X}}$.	
	Since the property is Zariski local on $X$, we can assume that each $D_i=\princ(f_i)$ is principal. Let $x$ be a closed point of $X$, by shrinking further, we can assume that $x\in D_i$ for all $i\in I$. The local equation of $D$ at $x$ is given by $\prod_{i\in I}f_i$ and the $D_i$'s being integral by assumption, the $f_i$'s are prime, hence irreducible. The local ring $\mathcal O_{X,x}$ is regular, hence factorial so the hypothesis that $D$ is sncd shows that set $\{f_i,i\in I\}$ can be ordered into a part of regular system of parameters at $x$, say $(x_1,\cdots,x_m)$. We complete it into a full regular system of parameters $(x_1,\cdots,x_n)$, which defines an étale morphism $X\to \mathbb A^n$. By affecting the integer $r_i=1$ to the $n-m$ last $x_i$'s, we don't change the stack of roots, in other words we get a commutative diagram with cartesian squares:

\begin{center}
	\begin{tikzcd}
Y \arrow[d] \arrow[r]           & \mathbb A^n \arrow[d] \arrow[rd]                        &                                           \\
\mathfrak X \arrow[d] \arrow[r] & {[\mathbb A^n/\mu_{\mathbf{r}}]} \arrow[d] \arrow[r] & {[\mathbb A^n/\mathbb G_m^n]} \arrow[d] \\
X \arrow[r]                     & \mathbb A^n \arrow[r]                                   & {[\mathbb A^n/\mathbb G_m^n]}          
\end{tikzcd}
\end{center}
Since $X\to \mathbb A^n$ is étale at $x$, by shrinking $X$ again, we can assume it is étale. Thus $Y\to \mathbb A^n$ is also étale, and so $Y$ is $k$-smooth. Since $Y\to\mathfrak X$ is an étale chart, we are done.

\end{proof}

\begin{remark}
\label{rem:sncd_implies_smooth}
\begin{enumerate}
	\item See also \cite[Proposition 3.9]{BLS:geo} for a slightly different approach.
	\item The claim would be wrong if one would only assume that $D$ is a ncd. To get a smooth stack of roots in the ncd case, one needs to use Olsson's definition, see \cite{BV:par}. This more elaborate formalism is out of the scope of the present article.

\end{enumerate}

\end{remark}

Let us now mention how to generalize the notion of (strict) normal crossings divisor to a locally noetherian Deligne-Mumford stack $\mathfrak{X}$. First, an effective Cartier divisor $ \mathfrak{D}\subset \mathfrak{X}$ is a closed substack that is an effective Cartier divisor in an étale chart (equivalently, such that the ideal sheaf $\mathcal I_\mathfrak{D} \subset \mathcal{O}_{\mathfrak{X}}$ is invertible).

\begin{definition}[]
\label{def:ncd-DM}
Let $\mathfrak{D}$ be an effective Cartier divisor on a locally noetherian Deligne-Mumford stack  $\mathfrak{X}$. We will say that:

\begin{enumerate}
	\item $\mathfrak{D}$ has normal crossings if this is true in an étale chart,
	\item $\mathfrak{D}$ has strict normal crossings if it has normal crossings and its irreducible components are regular.
\end{enumerate}

\end{definition}

\begin{remark}
\label{rem:ncd-DM}
\begin{enumerate}

\item On a scheme, according to \cite[Lemma 1.8.4]{G-M:tame}, the definition of sncd coincides with Definition \ref{def:ncd}.
\item It follows from the proof of Proposition \ref{prop:sncd_implies_smooth} that if $\mathfrak D$ is the Cartier divisor on $\sqrt[\mathbf{r}]{\mathbf{D}/X}$ whose irreducible components are $( \mathfrak{D}_i)_{i\in I}$, then $\mathfrak D$ has strict normal crossings.
\end{enumerate}

\end{remark}

\subsection{Locally free sheaves on the stack of roots and parabolic vector bundles}%
\label{sub:locally_free_sheaves_on_the_stack_of_roots_and_parabolic_vector_bundles}

In this section, we recall the main result of \cite[]{bor:rep}. We use the notations of the log-smooth context (\S \ref{ssub:logarithmic_context}).

Let us define parabolic vector bundles, following Carlos Simpson's formulation. We define the poset $\frac{1}{\mathbf r}\mathbb Z^I = \prod_{i\in I} \frac{\mathbb Z}{r_i} $ (with component-wise partial order), and identify it with the corresponding category. We write $\cdot ^{op}$ for the opposite category, and $\Vect(X)$ for the category of vector bundles on $X$.

\begin{definition}
\label{def:parabolic_bundle}
	
A \emph{parabolic vector bundle} on $(X, \mathbf D)$ with weights in
$\frac{1}{\mathbf r}\mathbb Z^I$ consists of
\begin{itemize}
\item the data of a functor $\mathcal E_\cdot \,:\, \left(\frac{1}{\mathbf r}
	\mathbb Z^I\right)^{op} \longrightarrow \Vect X$ and,

\item for each integral multi-index $\mathbf l$ in $\mathbf Z^I$, a natural
	transformation $p_{ \mathbf{l}	}: \mathcal E_{\cdot+\mathbf l} \simeq 
\mathcal E_{\cdot}\otimes_{\mathcal O_X} \mathcal O_X(-\mathbf l \cdot \mathbf D)$,
\end{itemize}
such that the following compatibility condition holds: for $\mathbf l \geq \mathbf 0$, the diagram of natural transformations

\begin{equation}
	\label{eq:parabolic_bundle}
	\begin{tikzcd}
\mathcal E_{\cdot+\mathbf l}  \arrow[rrd] \arrow[dd]                                          &  &                     \\
                                                                                              &  &  \mathcal E_{\cdot} \\
\mathcal E_{\cdot}\otimes_{\mathcal O_X} \mathcal O_X(-\mathbf l \cdot \mathbf D) \arrow[rru] &  &                    
\end{tikzcd}
\end{equation}
is commutative.

\end{definition}
For a more formal definition, see \cite[D\'efinition 2.1.2]{bor:rep}. We will most often omit the pseudo-periodicity isomorphism $p_\cdot$ from the notation and thus write $\mathcal E_\cdot$ instead of $(\mathcal E_\cdot,p_\cdot)$.  We denote by $\Par_{\frac{\mathbf 1}{\mathbf r}}(X,\mathbf D)$ the corresponding category.

\begin{remark}
\label{rem:parabolic_bundle}
The existence of the pseudo-periodicity isomorphisms implies that a parabolic bundle is determined, up to isomorphism, by its restriction to the fundamental domain $\frac{1}{\mathbf r}\mathbb Z^I\cap [0,1[^I$. The fact that $\textbf{D}$ is a family of (strict) normal crossings divisors implies much more, namely according to the forthcoming Lemma \ref{lem:parabolic_bundle_restricted_axes} a parabolic bundle is even determined by its restriction to the axes in $\frac{1}{\mathbf r}\mathbb Z^I\cap [0,1[^I$.

\end{remark}

\begin{lemma}
\label{lem:parabolic_bundle_restricted_axes}
Let $\mathcal E_\cdot$ be an object of $\Par_{\frac{\mathbf 1}{\mathbf r}}(X,\mathbf D)$ and 
$ \mathbf{l}$, $ \mathbf{l'}$ in $\mathbb Z^I$ such that $\mathbf{l}\leq \mathbf{l'} \leq \mathbf{l}+\mathbf{r}$. Then if $(e_i)_{i\in I}$ denotes the canonical basis of $\mathbb Z^I$,  we have as subsheaves of  \( \mathcal{E}_{\frac{	\mathbf{l}}{ \mathbf{r}}} \):
\[ \mathcal{E}_{\frac{	\mathbf{l'}}{ \mathbf{r}}} = \bigcap_{i\in I} \mathcal{E}_{\frac{	\mathbf{l}+(l'_i-l_i)e_i}{ \mathbf{r}}} \]

\end{lemma}

\begin{proof}
	The inclusion $\mathcal{E}_{\frac{	\mathbf{l'}}{ \mathbf{r}}} \subset \bigcap_{i\in I} \mathcal{E}_{\frac{	\mathbf{l}+(l'_i-l_i)e_i}{ \mathbf{r}}} $ is clear. For the other direction, let us first remark that for each $i\in I$ we have  
	$$\mathcal{E}_{\frac{	\mathbf{l}+(l'_i-l_i)e_i}{ \mathbf{r}}} \subset\mathcal{E}_{\frac{	\mathbf{l'}-\mathbf{r}}{ \mathbf{r}}+e_i} =  \mathcal{E}_{\frac{	\mathbf{l'}-\mathbf{r}}{ \mathbf{r}}} (-D_i) \; .$$
It follows that $\bigcap_{i\in I} \mathcal{E}_{\frac{	\mathbf{l}+(l'_i-l_i)e_i}{ \mathbf{r}}} \subset \bigcap_{i\in I} \mathcal{E}_{\frac{	\mathbf{l'}-\mathbf{r}}{ \mathbf{r}}} (-D_i)$. But there is also a natural inclusion
$$\mathcal{E}_{\frac{	\mathbf{l'}}{ \mathbf{r}}}=\mathcal{E}_{\frac{	\mathbf{l'}-\mathbf{r}}{ \mathbf{r}}}(-\sum_{i\in I}D_i) \subset \bigcap_{i\in I} \mathcal{E}_{\frac{	\mathbf{l'}-\mathbf{r}}{ \mathbf{r}}} (-D_i)\; . $$
The fact that $\mathcal{E}_{\frac{	\mathbf{l'}-\mathbf{r}}{ \mathbf{r}}}$ is locally free and that the local equations of the $D_i$'s are coprime shows that this last inclusion is in fact an equality, which proves the result.
\end{proof}

\begin{definition}[]
\label{def:functors_F_G}

\begin{enumerate}
\item 

To each vector bundle $\mathcal F$ on $\mathfrak X$, one associates a parabolic 
vector bundle $\widehat{\mathcal F}_\cdot$ on $(X,\mathbf D)$ with weights in 
$\frac{1}{\mathbf r}\mathbb Z^I$ in the following way: if $\mathbf l$ belongs to 
$\mathbb Z^I$, one defines $\widehat{\mathcal F}_{\frac{\mathbf l}{\mathbf r}} =\pi_* 
\left(\mathcal F\otimes_{\mathcal O_{\mathfrak X}} \mathcal O_{\mathfrak X}( -\mathbf 
l \mathbf{\mathfrak D  })\right)$.

\item 
Conversely, let $\mathcal E_\cdot$ be an object in $\Par_{\frac{1}{\mathbf r}}(X,\mathbf D)$. One associates to this parabolic vector bundle a vector bundle on the stack of roots defined by:

$$\widehat{\mathcal E_\cdot}=\int^{\frac{1}{\mathbf r}\mathbb Z^I}\pi^*\mathcal E_\cdot \otimes \mathcal O_{\mathfrak X}( \cdot  \mathbf r \mathbf{\mathfrak D ) } $$
where $\int^{\frac{1}{\mathbf r}\mathbb Z^I}$ stands for the coend\footnote{See \cite[IX \S6]{MCL:categories} or \cite[Appendice B]{bor:corr} for a summary}.
\end{enumerate}
\end{definition}

\begin{theorem}[{\cite[Th\'eor\`eme 2.4.7]{bor:rep}}]
\label{thm:description}
The functors  $\mathcal E_\cdot \mapsto \widehat{\mathcal E_\cdot}$ and $\mathcal F \mapsto \widehat{\mathcal F}_\cdot$ define inverse equivalences between the categories $\Par_{\frac{1}{\mathbf r}}(X,\mathbf D)$ and $\Vect(\sqrt[\mathbf{r}]{\mathbf{D}/X})$.
\end{theorem}

\section{Connections on Deligne-Mumford stacks}%
\label{sec:connections_on_dm_stacks}

\subsection{Holomorphic connections}%
\label{sub:holomorphic_connections}

 The literature on the subject is 
%much more 
 sparse, even if the notion is widely used, especially in the geometric Langlands program. Our main reference is \cite[Chapter 2]{ols:cry} where the -simplest- point of view of sheaves on the small étale site is used.

 \subsubsection{Definition on schemes and internal operations}%
 \label{ssub:definition_on_schemes_and_internal_operations}

 Let us start by considering $k$-schemes $X$, $X'$ ... There are many equivalent definitions of a connection $\nabla$ on a vector bundle $\mathcal E$ on $X$, but we considering first the most frequent one, Koszul's definition: a connection is a $k$-linear morphism $\nabla:\mathcal E\to \mathcal E\otimes_{\mathcal O_X} \Omega^1_{X/k}$, satisfying Leibniz rule, that is $\nabla(fs)=f\nabla(s)+s\otimes df$ for $f\in \mathcal O_X$ and $s\in \mathcal E$.

The category of pairs $(\mathcal E,\nabla)$ is endowed with:
\begin{itemize}
	\item a tensor product given by $(\mathcal E,\nabla)\otimes (\mathcal E',\nabla')= 
(\mathcal E\otimes_{\mathcal O_X}\mathcal E',\nabla\otimes\nabla' )$ where $(\nabla\otimes \nabla')(s\otimes s')=\nabla s\otimes s'+s\otimes \nabla' s'$,

 \item an internal Hom defined by $\HOM((\mathcal E,\nabla),(\mathcal E',\nabla'))=(\HOM(\mathcal E,\mathcal E'), \nabla^{\HOM})$ where $\nabla^{\HOM}$ verifies  $\nabla'(\phi(s))=(\phi\otimes\id)(\nabla(s))+\nabla^{\HOM}(\phi)(s)$.

\end{itemize}
	 In particular, one can define the dual of $(\mathcal E,\nabla)$ as 
$(\mathcal E,\nabla)^{\vee}=\HOM((\mathcal E,\nabla),(\mathcal O_X,d))$.

\subsubsection{Functoriality of connections}%
\label{ssub:functoriality_of_connections}

The following result is well-known, but due to lack of a proper reference, we sketch a proof.

\begin{lemma}
\label{lem:pull_back_connection}
Let $f:X'\to X$  be a morphism of $k$-schemes, $\mathcal E$ a vector bundle on $X$ endowed with a connection $\nabla$. There exists a unique connection $f^*\nabla$ on $f^*\mathcal E$ such that $f^*\nabla(f^* s)=f^*(\nabla(s))$ for all $s\in \mathcal E$.
\end{lemma}

\begin{proof}
	The uniqueness follows from Leibniz rule, as the sections $f^*s$ generate $f^*\mathcal E$ locally. To show the existence, it is thus enough to show it Zariski locally. But then $\nabla$ is given by a matrix of differential forms, and one defines $f^*\nabla$ thanks to the matrix obtained by pulling-back each form individually.
\end{proof}

\subsubsection{Definitions of holomorphic connection on a Deligne-Mumford stack}%
\label{ssub:definitions_of_holomorphic_connection_on_a_deligne_mumford_stack}

As for quasi-coherent sheaves 
, one can use different  - but equivalent - points of view to define an holomorphic connection on a Deligne-Mumford stack $\mathfrak X$ defined over a field $k$.

\begin{enumerate}
	\item It is quite natural to use the (small) étale site $\Et(\mathfrak X)$. This is Martin Olsson's point of view in \cite[2.2.19-23]{ols:cry}. On this site, the sheaf $\Omega_{\mathfrak X/k}^1$ is defined by $\Omega_{\mathfrak X/k}^1(T\to \mathfrak X) = \Gamma(T, \Omega_{T/k}^1)$, and there is a canonical derivation $d:\mathcal O_{\mathfrak X}\to \Omega_{\mathfrak X/k}^1$. Then the definition is the usual one: a connection $(\mathcal E,\nabla)$ is a locally free sheaf on $\Et(\mathfrak X)$ endowed with a $k$-linear morphism $\nabla:\mathcal E\to \mathcal E\otimes_{\mathcal O_{\mathfrak X}} \Omega^1_{\mathfrak{X}/k}$, satisfying Leibniz rule.

		We will rather use the close, and equivalent, point of view of giving the following data: 

		\begin{itemize}
			\item for each étale morphism from a scheme $t:T\rightarrow \mathfrak X$ of a locally free sheaf endowed with a connection $(\mathcal E_{(T,t)}, \nabla_{(T,t)})$ on $T$,
			\item and for each $2$-morphism $(f,f^b):(T',t')\to (T,t)$:
\begin{center}
	\begin{tikzcd}
                                     & T \arrow[rd, "t"] \arrow[d, "f^b"', Rightarrow] &             \\
T' \arrow[rr, "t'"'] \arrow[ru, "f"] & {}                                              & \mathfrak X
\end{tikzcd}
\end{center}
of an isomorphism $\rho_{(f,f^b)}:f^* \mathcal E_{(T,t)} \to \mathcal E_{(T',t')}$, compatible with the connections,
this data verifying the usual cocycle condition. 
\end{itemize}

	\item One can also use the point of view of groupoids: let $U\to \mathfrak X$  be étale chart, and  $(s,t):R\rightrightarrows U$ the corresponding groupoid.

		A connection on this groupoid consists of triple $(\mathcal F, \alpha,\nabla)$ where

		\begin{itemize}
			\item $\mathcal F$ is a vector bundle endowed with a connection $\nabla$ on $U$,
			\item $\alpha$ is an isomorphism $\alpha:t^*\mathcal F \simeq s^*\mathcal F$, compatible with $t^*\nabla$ and $s^*\nabla$,
		\end{itemize}
		again submitted to the usual cocycle condition.

\end{enumerate}

\subsection{Atiyah's exact sequence}%
\label{sub:atiyah_s_exact_sequence}

As is well known, connections can also be described within the $\mathcal O_X$-linear world. This is a precious point of view as it will considerably simplify some proofs. We recall very briefly the definition of Atiyah's extension following the exposition in \cite[\S1]{bost:atiyah}.

We start with the case of schemes, working over an arbitrary basis $S$.

Let $X$ be a separated $S$-scheme. Obviously the notion of a connection over $S$ still makes sense. Let $X^{(1)}$ be the first infinitesimal neighborhood  of the diagonal $\Delta : X \rightarrow X\times_S X$, let $i: X\to X^{(1)}$ be the canonical closed immersion, and for $j\in \{1,2\}$, let $q_j:X^{(1)}\rightarrow X $ be the composition of $X^{(1)} \rightarrow X\times_S X$ with $\pr_j:X\times_S X\rightarrow X$.

If $\mathcal I$ is the ideal defined by $\Delta$, then we identify $ \frac{\mathcal I }{\mathcal I^2}$ with $i_*\Omega ^1_{X/S}$, and the canonical  connection $d: \mathcal O_X\rightarrow  \Omega ^1_{X/S}$ with $i_* \mathcal O_X \rightarrow \frac{\mathcal I }{\mathcal I^2}$ given by $f\mapsto q_2^*f-q_1^*f$.

Let $\mathcal E$ be a vector bundle on $X$. By tensoring the canonical exact sequence $$0\rightarrow i_*\Omega ^1_{X/S} \to \frac{\mathcal O_{X\times_S X}}{\mathcal I^2} \rightarrow i_* \mathcal O_X  \rightarrow 0$$ with $q_2^*\mathcal E$ and then applying ${q_1}_*$, we get a canonical ($\mathcal O_X$-linear) exact sequence of sheaves:

\[ 0\rightarrow \Omega ^1_{X/S}\otimes_{\mathcal O_X}\mathcal E  \rightarrow P^1_{X/S}(\mathcal E) \rightarrow \mathcal E \rightarrow 0 \]
where $P^1_{X/S}(\mathcal E)={q_1}_*q_2^*\mathcal E$ is the sheaf of principal parts of $\mathcal E$. This sequence is known as \emph{Atiyah's extension} of $\mathcal E$.

Since $i:X\rightarrow X^{(1)}$ is an homeomorphism, it follows that  $P^1_{X/S}(\mathcal E)={q_2}_*q_2^*\mathcal E$ as sheaves of $\mathcal O_S$-modules, thus $q_2^\#:\mathcal E\to {q_2}_*q_2^*\mathcal E$ gives an $\mathcal O_S$-linear splitting of Atiyah's extension. Now if $\mathcal \alpha: \mathcal E \to P^1_{X/S}(\mathcal E)$ is another $\mathcal O_S$-linear splitting, then $\alpha$ is $\mathcal O_X$-linear if and only if $\nabla=\alpha-q_2^\#: \mathcal E \rightarrow \Omega ^1_{X/S}\otimes_{\mathcal O_X}\mathcal E$ verifies Koszul's condition. One concludes that the connections on $\mathcal E$ are in one to one correspondence with the $\mathcal O_X$-linear splittings of Atiyah's extension. With some additional care, one shows that this holds for an arbitrary $S$-scheme, see \cite[Proposition 2.9]{bert-ogus:notes}.

It is clear that the formation of  $P^1_{X/S}(\mathcal E)$ commutes with an étale base change $X'\to X$. As a consequence, we can define Atiyah's extension for a vector bundle $\mathcal E$ on a Deligne-Mumford stack $\mathfrak X/S$, and the correspondence above still holds. However, we will rather need the logarithmic version, that we construct directly (\S \ref{ssub:logarithmic_connexions_on_deligne_mumford_stacks}).

\subsection{Logarithmic connections}%
\label{sub:logarithmic_connections}

In this section, we use the notations of the log-smooth context (\S \ref{ssub:logarithmic_context}).

\subsubsection{Logarithmic differentials}%
\label{ssub:logarithmic_differentials}

We start by revisiting the notion of logarithmic differentials on schemes. The classical and most intuitive way of defining $\Omega^1_{X/k}(\log(D))$ consists of viewing it as the subsheaf of  $\Omega^1_{X/k}(\star D)$ locally generated by the forms that are logarithmic along $D$ (see for instance \cite[Definition 2.1]{E-W:book}). 

However, the local nature of this definition makes it a bit cumbersome to manipulate when one has to deal with algebraic stacks. Instead, we use Martin Olsson's insight that logarithmic differentials on a stack $\mathfrak X\to [\mathbb A^1/\mathbb G^m]$ can be defined as $\Omega^1_{\mathfrak X/[\mathbb A^1/\mathbb G^m]}$. As $[\mathbb A^1/\mathbb G^m]$ is not a Deligne-Mumford stack, using this directly as a definition would imply the use of the lisse-étale (or fppf) site, that we want to avoid. For representable morphisms $\mathfrak X\to [\mathbb A^1/\mathbb G^m]$ though, it is easy to spell out the meaning of Olsson's definition in a chart, and this point of view gives a global definition of logarithmic differentials that seems very useful even for schemes. Remember that we have defined the $\mathbb{G}_m^I$-torsor $p_D:T_D\rightarrow X$ in \S \ref{ssub:canonical_flat_presentation}.

\begin{definition}[Martin Olsson]
\label{def:log_diff_Ollson}
The sheaf of logarithmic differentials is defined as  
$\Omega^1_{X/k}(\log(D))={p_D}_*^{\mathbb G_m^I} \Omega^1_{T_D/\mathbb A^I}$.
\end{definition}

Using this definition, one proves easily the following classical fact:
\begin{proposition}
\label{prop:exact_sequence_residues}
There  is a natural exact sequence:
\[ 0\rightarrow \Omega^1_{X/k}\rightarrow \Omega^1_{X/k}(\log(D))\xrightarrow{\res} \oplus_{i\in I}{i_{D_i}}_*\mathcal O_{D_i}\rightarrow 0 \] 
where for each $i\in I$, $i_{D_i}$ stands for the closed immersion of $D_i$ in $X$.
\end{proposition}

Let us now describe how logarithmic differentials behave functorially. As we don't need the full power of logarithmic geometry, we will work with an ad hoc notion of log-scheme.

\begin{itemize}
	\item A log-scheme is a couple $(X,(D_i)_{i\in I})$ where $X$ is a smooth $k$-scheme and $(D_i)_{i\in I}$ is a finite family of distinct effective integral Cartier divisors such that the divisor $D=\cup_{i\in I}D_i$ is a sncd.
	\item A morphism between two such log-schemes $(X',(D'_i)_{i\in I})$ and $(X,(D_i)_{i\in I})$, indexed by the same finite set $I$, is a couple $(f, (r_i))_{i\in I})$ consisting of a flat morphism $f: X'\to X$ and a family of non-negative integers $(r_i)_{i\in I}$ such that $f^* D_i=r_iD'_i$ for each $i\in I$.
\end{itemize}  

As the second part of the data of a morphism is redundant, we will frequently omit it. It is clear from the definitions that to such a morphism is associated a canonical morphism $f^*\Omega^1_{X/k}(\log(D))\to \Omega^1_{X'/k}(\log(D'))$.

\begin{definition}[]
\label{def:log-etale}
A morphism $(f, (r_i))_{i\in I}): (X',(D'_i)_{i\in I})\to(X,(D_i)_{i\in I})) $ is log-étale if 
\begin{enumerate}
	\item the integer $r_i$ is invertible in $k$ for each $i\in I$,
	\item the associated morphism $X'\to \sqrt[\mathbf{r}]{{\mathbf{D}/X}}$ is étale.
\end{enumerate}

\end{definition}

\begin{lemma}
\label{lem:Kummer}
Let $(f, (r_i))_{i\in I}): (X',(D'_i)_{i\in I})\to(X,(D_i)_{i\in I}))$ be log-étale morphism. Then the canonical morphism $f^*\Omega^1_{X/k}(\log(D))\to \Omega^1_{X'/k}(\log(D'))$ is an isomorphism.
\end{lemma}

\begin{proof}
This follows from the two well-known facts: formation of differentials commutes with arbitrary base change, and an étale morphism is unramified.
\end{proof}

Finally, we indicate briefly how to define logarithmic differentials in the stacky log-smooth context (\S \ref{ssub:stacky_context}). Let $\mathfrak{X}$ be  a $k$-smooth Deligne-Mumford stack  endowed with a finite family $( \mathfrak{D}_i)_{i\in I}$ of distinct effective integral Cartier divisors such that the divisor $\mathfrak{D}=\cup_{i\in I}\mathfrak{D}_i$ is a sncd.

Let  $(f,f^b):(T',t')\to (T,t)$ be a $2$-morphism between objects of the small étale site of $\mathfrak X$ as in \S 
\ref{ssub:definitions_of_holomorphic_connection_on_a_deligne_mumford_stack}. Then Lemma \ref{lem:Kummer} implies that there is a canonical isomorphism $f^*\Omega^1_{T/k}(\log(t^*\mathfrak D))\to \Omega^1_{T'/k}(\log(t'^*\mathfrak D'))$. The cocycle condition is verified, so this defines a sheaf $\Omega^1_{ \mathfrak{X}/k}(\log( \mathfrak{D}))$. 

However, for stack of roots, a more explicit approach is available. Namely more generally for stacks such that $T_{\mathfrak D}$ is a scheme, we can simply use Martin Olsson's definition (Definition \ref{def:log_diff_Ollson}) as it is.

There is a straightforward but useful definition of a log-stack (that is, a couple $(\mathfrak{X},( \mathfrak{D}_i)_{i\in I})$ as above) and log-étale morphism between log-stacks.  For instance, if $\mathfrak X= \sqrt[\mathbf{r}]{(\mathbf{D}/X}$ is a stack of roots and $( \mathfrak{D}_i)_{i\in I}$ is the family of roots, then $(\mathfrak{X},( \mathfrak{D}_i)_{i\in I})$ is a log-stack (see Remark \ref{rem:ncd-DM}) and the log-morphism $(\pi,(r_i)_{i\in I}): (\mathfrak{X},( \mathfrak{D}_i)_{i\in I}) \rightarrow (X,(D_i)_{i\in I})$ is tautologically log-étale.

\subsubsection{Logarithmic connections}%
\label{ssub:logarithmic_connections}

A \emph{logarithmic connection} $\nabla$ on a vector bundle $\mathcal E$ on $X$ is a $k$-linear morphism $\nabla:\mathcal E\to \mathcal E\otimes_{\mathcal O_X} \Omega^1_{X/k}(\log(D))$ satisfying Leibniz rule. The corresponding category is denoted by $\Con(X,D)$. 

There is a useful formula for the residue of a tensor product of two logarithmic connections \( (\mathcal{E},\nabla) \) and \( (\mathcal{E'},\nabla') \)  :

 \[ \res_{D_i}(\nabla \otimes \nabla') = \res_{D_i}(\nabla) \otimes \id_{ \mathcal{E'}_{|D_i}}  + \id_{ \mathcal{E}_{|D_i}}\otimes \res_{D_i}(\nabla') \] 

One fact of paramount importance for stating the forthcoming correspondence between \parabolic connections and holomorphic connections on the stack of roots (see \S \ref{sec:the_correspondence}) is the existence of a canonical logarithmic connection on the ideal sheaf $\mathcal I_{D_i}= \mathcal O_X(-D_i)$. Let us describe its construction in the terms of  \S \ref{ssub:logarithmic_differentials}. We first notice that the canonical holomorphic connection $d:\mathcal{O}_{T_D}\to \Omega^1_{T_D/\mathbb{A}^I}$ is $\mathbb{G}_m^I$-equivariant. Let $(a_D)_i$ be the $i$-th component of $a_D$, this is a global equation of the principal divisor $p_D^*D_i$. Since $d((a_D)_i)=0$ in $\Omega^1_{T_D/\mathbb{A}^I}$, it follows that $d((a_D)_i\mathcal{O}_{T_D})  \subset (a_D)_i\Omega^1_{T_D/\mathbb{A}^I}$. By applying the functor $(p_D)_*^{ \mathbb{G}_m^I}$ to the restriction $d:(a_D)_i\mathcal{O}_{T_D}  \rightarrow  (a_D)_i\Omega^1_{T_D/\mathbb{A}^I}$ we get a logarithmic connection 
$$d(-D_i): \mathcal O_X(-D_i) \rightarrow \mathcal O_X(-D_i)\otimes_{\mathcal O_X} \Omega^1_{X/k}(\log(D))\; .$$

\begin{remark}
\label{rem:canonical connection}
By construction, the morphism \(    \mathcal O_X(-D_i) \rightarrow  \mathcal O_X \) is compatible with the connections $d(-D_i)$ and $d$, and this caracterizes $d(-D_i)$ uniquely.
\end{remark}

\begin{lemma}[{\cite[Lemma 2.7]{E-W:book}}]
\label{lem:canonical_connection}
Let $B=\sum_{i\in I}\mu_i D_i$ be a Cartier divisor with support in $D$. There exists a canonical logarithmic connection: 
$$d(B): \mathcal O_X(B) \rightarrow \mathcal O_X(B)\otimes_{\mathcal O_X} \Omega^1_{X/k}(\log(D))$$
characterized by $d(B)(\prod_{i\in I}x_i^{-\mu_i})=-\prod_{i\in I}x_i^{-\mu_i}\cdot\sum_{i\in I}\mu_i \frac{dx_i}{x_i} $, where $x_i$ is a local equation of $D_i$. In particular, $\res_{D_i} (d(B))=-\mu_i  \id$.

\end{lemma}

\begin{proof}
	When $B=-D_i$, this is a consequence of the discussion above. The general case follows by using the dual of a connection and the tensor product of two connections.
\end{proof}

In particular, using again the tensor product of two connections, we can twist an arbitrary logarithmic connection $(\mathcal E,\nabla)$ by a divisor $B=\sum_{i\in I}\mu_i D_i$, the result will be denoted by $(\mathcal E(B),\nabla(B))$. From \cite[Lemma 2.7]{E-W:book}, we borrow the following formula, which we will use extensively:
$$\res_{D_i} (\nabla(B))= \left(\res_{D_i} (\nabla)\right)(B)-\mu_i  \id$$
and which follows from the formula giving the residue of a tensor product of logarithmic connections.

\begin{remark}
\label{rem:push-forward-logarithmic-connection}
Let $f: (X',(D'_i)_{i\in I})\to(X,(D_i)_{i\in I})) $ be a morphism of log-schemes. The pull-back $f^*:\Con(X,D) \rightarrow \Con(X',D')$ is well defined (this is analogous to Lemma \ref{lem:pull_back_connection}). If $f$ is finite and log-étale then the push-forward $f_*:\Con(X',D') \rightarrow \Con(X,D)$ is also well defined (thanks to projection formula and Lemma \ref{lem:Kummer}).

\end{remark}

\subsubsection{Logarithmic Atiyah exact sequence}%
\label{ssub:logarithmic_atiyah_exact_sequence}

In this section, we denote by $\mathcal E$ a vector bundle on $X$.

\begin{definition}[]
\label{def:logarithmic_principal_parts}
The sheaf of \emph{logarithmic principal parts} (with respect to $D$) is the sheaf:
\[ P^1_{(X,D)/k}( \mathcal{E}) := {p_D}_*^{\mathbb G_m^I} \left( P^1_{_{T_D/\mathbb A^I}}(p_D^*\mathcal E) \right)  \] 
\end{definition}

\begin{lemma}
\label{lem:logarithmic_atiyah_exact_sequence}
There is a natural exact sequence:
\[ 0\rightarrow \Omega ^1_{X/k}(\log D) \otimes_{\mathcal O_X}\mathcal E  \rightarrow P^1_{(X,D)/k}( \mathcal{E}) \rightarrow \mathcal E \rightarrow 0 \] 
\end{lemma}

\begin{proof}
This is the image of the standard Atiyah sequence 
$$0 \rightarrow \Omega^1_{T_D/\mathbb A^I}\otimes p_D^*\mathcal E \rightarrow P^1_{_{T_D/\mathbb A^I}}(p_D^*\mathcal E) \rightarrow p_D^*\mathcal E \rightarrow 0$$ by the exact functor ${p_D}_*^{\mathbb G_m^I}$.

\end{proof}

\begin{lemma}
\label{lem:sections-logarithmic-Atiyah-extension}
There is a natural bijection between logarithmic connections on $\mathcal E$ and sections of the logarithmic Atiyah exact sequence.
\end{lemma}

\begin{proof}
	Since this holds for holomorphic connections (see \S \ref{sub:atiyah_s_exact_sequence}), it is enough to show that logarithmic connections $\nabla:\mathcal E\to \mathcal E\otimes_{\mathcal O_X} \Omega^1_{X/k}(\log(D))$ correspond to connections  $\nabla_D: p_D^*\mathcal E \rightarrow   \Omega^1_{T_D/\mathbb A^I}\otimes p_D^*\mathcal E$ that respect the natural $\mathbb Z^I$-graduations. Starting from $\nabla_D$ we put $\nabla={p_D}_*^{\mathbb G_m^I}\nabla_D$, conversely starting from $\nabla$ we can define  $\nabla_D=p_D^*\nabla$ as in Lemma \ref{lem:pull_back_connection}. It is clear that these constructions are inverse of each other.
\end{proof}

As an immediate consequence, we get the following:

\begin{corollary}
\label{cor:sections-logarithmic-Atiyah-extension}
There is a natural equivalence of categories between
\begin{itemize}
	\item the category $\Con(X,D)$ of logarithmic connections,
	\item the category $\Sec(X,D)$, whose objects are couples $(\mathcal E,\alpha)$, where $\mathcal E$ is a vector bundle on $X$ and $\alpha$ is a section of the logarithmic Atiyah exact sequence of $\mathcal E$, with obvious morphisms.
\end{itemize}
\end{corollary}

The existence of twists of logarithmic connections can now be explained in a somewhat more natural way.

\begin{lemma}
\label{lem:twist_Atiyah_exact_sequence}
Let $B=\sum_{i\in I}\mu_i D_i$ be a Cartier divisor with support in $D$, and $\mathcal{E}$ be a vector bundle on $X$.
The logarithmic Atiyah exact sequence of $\mathcal E(B)$ identifies with the twist of the logarithmic Atiyah exact sequence by $B$. In other words, there is a natural isomorphism $P^1_{(X,D)/k}( \mathcal{E}(B))\simeq P^1_{(X,D)/k}( \mathcal{E})(B)$, compatible with the morphisms in the logarithmic Atiyah exact sequences. 
\end{lemma}

\begin{proof}
	Reasoning on $T_D$, is is enough to see that there is a natural isomorphism $P^1_{_{T_D/\mathbb A^I}}(p_D^*\mathcal E(B))\simeq  P^1_{_{T_D/\mathbb A^I}}(p_D^*\mathcal E)(p_D^*(B))$ that is \( \mathbb G_m^I \)-equivariant, or in other words, of degree $0$ with respect to the natural $ \mathbb{Z}^I$-graduations. But the canonical trivialisation $\mathcal{O}_{T_D}\simeq  \mathcal{O}_{T_D}(p_D^*(B))$ is of degree $(\mu_i)_{i\in I}$, and gives rise to isomorphisms $P^1_{_{T_D/\mathbb A^I}}(p_D^*\mathcal E)\simeq P^1_{_{T_D/\mathbb A^I}}(p_D^*\mathcal E(B))$ and $P^1_{_{T_D/\mathbb A^I}}(p_D^*\mathcal E)\simeq P^1_{_{T_D/\mathbb A^I}}(p_D^*\mathcal E)(p_D^*(B)$, both of degree $(\mu_i)_{i\in I}$, hence the result.

\end{proof}

\subsubsection{Logarithmic connexions on Deligne-Mumford stacks}%
\label{ssub:logarithmic_connexions_on_deligne_mumford_stacks}

Assume now that we are in the stacky log-smooth context (\S \ref{ssub:stacky_context}). 
We can define logarithmic connections on the log-stack $(\mathfrak X,\mathfrak D)$ just as we did for holomorphic connections (\S\ref{ssub:definitions_of_holomorphic_connection_on_a_deligne_mumford_stack}), using the derivation  $d:\mathcal O_{\mathfrak X} \rightarrow \Omega_{\mathfrak X/k}^1(\log(\mathfrak D))$ instead of $d:\mathcal O_{\mathfrak X} \rightarrow  \Omega_{\mathfrak X/k}^1$.

If we assume that $T_{\mathfrak D}=\mathfrak X\times_{[ \mathbb{A}^I/ \mathbb{G}_m^I]}\mathbb{A}^I$ is a scheme, then the constructions of the logarithmic Atiyah exact sequence associated to a vector bundle $\mathcal E$ on $\mathfrak X$ (Lemma \ref{lem:logarithmic_atiyah_exact_sequence}) and the interpretation of its sections as logarithmic connections on $\mathcal E$ (Lemma \ref{lem:sections-logarithmic-Atiyah-extension}, Corollary \ref{cor:sections-logarithmic-Atiyah-extension}) also hold in this context.

\section{The correspondence}%
\label{sec:the_correspondence}

In this section, we use the notations of the log-smooth context (\ref{ssub:logarithmic_context}).

\subsection{Parabolic connections}%
\label{sub:parabolic_connections}

\begin{definition}[]
\label{def:pre-parabolic_connection}
 \begin{itemize}
A \emph{\preparabolic connection} on  $(X, \mathbf D)$ with weights in $\frac{1}{\mathbf r}\mathbb Z^I$ consists of

\item the data of a functor $\mathcal E_\cdot \,:\, \left(\frac{1}{\mathbf r}
	\mathbb Z^I\right)^{op} \longrightarrow \Con(X,D)$ and,

\item the structure of a parabolic vector bundle (Definition \ref{def:parabolic_bundle}) on the underlying functor  $\left(\frac{1}{\mathbf r}	\mathbb Z^I\right)^{op} \longrightarrow \Vect X$.

\end{itemize}
\end{definition}

\begin{remark}
\label{rem:pre-parabolic_connection}

One could define a \preparabolic connection directly by substituting $\Vect X$ by $\Con(X,D)$ in Definition \ref{def:parabolic_bundle}. This is a priori a stronger requirement since the pseudo-periodicity isomorphism is then assumed to be compatible with the connections (a statement that makes sense as   $\mathcal E_{\cdot}\otimes_{\mathcal O_X} \mathcal O_X(-\mathbf l \cdot \mathbf D)$ is naturally endowed with a logarithmic connection, see Lemma \ref{lem:canonical_connection}). But the commutativity of diagram \eqref{eq:parabolic_bundle} shows that this property is in fact automatically realized with our current definition.

\end{remark}

If $ \mathbf{l'}\geq \mathbf{l}$, for each $i\in I$ the induced morphism   $(\mathcal E_{ \frac{\mathbf{l'}}{\mathbf{r}}})_{|D_i} \rightarrow (\mathcal E_{ \frac{\mathbf{l}}{\mathbf{r}}})_{|D_i}$ is compatible with the residues $\res_{|D_i}(\nabla_{ \frac{\mathbf{l'}}{\mathbf{r}}})$ and 
$\res_{|D_i}(\nabla_{\frac{\mathbf{l}}{\mathbf{r}}})$. %Now l
Let us denote by $(e_i)_{i\in I}$ the canonical basis of $\mathbb Z^I$. The morphism above for $\mathbf{l'}=\mathbf{l}+e_i$ has a canonical factorization 

\[ \frac{\mathcal E_{ \frac{\mathbf{l}+e_i}{\mathbf{r}}} }{\mathcal O_X(-D_i)\mathcal E_{ \frac{\mathbf{l}+e_i}{\mathbf{r}}}} \twoheadrightarrow \frac{\mathcal E_{ \frac{\mathbf{l}+e_i}{\mathbf{r}}} }{\mathcal O_X(-D_i)\mathcal E_{ \frac{\mathbf{l}}{\mathbf{r}}}}  \hookrightarrow  \frac{\mathcal E_{ \frac{\mathbf{l}}{\mathbf{r}}} }{\mathcal O_X(-D_i)\mathcal E_{ \frac{\mathbf{l}}{\mathbf{r}}}} \]
and thus by compatibility of the residue morphisms, the middle term \( \mathcal E_{ \frac{\mathbf{l}+e_i}{\mathbf{r}}}  \) is stable by $\res_{|D_i}(\nabla_{\frac{\mathbf{l}}{\mathbf{r}}})$, hence we get an induced morphism on the quotient \(  \frac{\mathcal E_{ \frac{\mathbf{l}}{\mathbf{r}}} }{\mathcal E_{ \frac{\mathbf{l}+e_i}{\mathbf{r}}} }\) that we again denote by  $\res_{|D_i}(\nabla_{\frac{\mathbf{l}}{\mathbf{r}}})$.

\begin{definition}[]
\label{def:parabolic_connection}
A \emph{\parabolic connection} on  $(X, \mathbf D)$ with weights in $\frac{1}{\mathbf r}\mathbb Z^I$ is a \preparabolic connection $(\mathcal E_{\cdot},\nabla_{\cdot})$ such that for each $i\in I$ the induced morphism 
$\res_{|D_i}(\nabla_{\frac{\mathbf{l}}{\mathbf{r}}})$ on \(  \frac{\mathcal E_{ \frac{\mathbf{l}}{\mathbf{r}}} }{\mathcal E_{ \frac{\mathbf{l}+e_i}{\mathbf{r}}} }\) is equal to \( \frac{l_i}{r_i} \id \). 
\end{definition}

The corresponding category will be denoted by $\Par\Con_{\frac{\mathbf 1}{\mathbf r}}^{st}(X,\mathbf D)$.

\subsection{Parabolic connections as sections of the parabolic Atiyah exact sequence}%
\label{sub:pre_parabolic_connections_as_sections_of_the_parabolic_atiyah_exact_sequence}

There is a neat interpretation of a \preparabolic connection within the parabolic world.
Namely, let $\mathcal E_{\cdot}$ be  a parabolic bundle.  By Lemma \ref{lem:twist_Atiyah_exact_sequence} $P^1_{(X,D)/k}(\mathcal E_{\cdot})$ admits a natural parabolic structure, and fits into the following parabolic Atiyah exact sequence built componentwise:
\[ 0 \rightarrow  \mathcal E_{\cdot}\otimes \Omega^1_{X/S} (\log(D)) \rightarrow P^1_{(X,D)/k} (\mathcal E_{\cdot})  \rightarrow \mathcal E_{\cdot} \rightarrow 0  \; , \]

\begin{lemma}
\label{lem:sections-parabolic-Atiyah-extension}
There is a natural bijection between parabolic connections $(\mathcal E_{\cdot},\nabla_{\cdot})$ with underlying parabolic bundle $\mathcal E_\cdot $ and sections of the parabolic Atiyah exact sequence of $\mathcal E_\cdot $.
\end{lemma}

\begin{proof}
	As the bijection between logarithmic connections and sections of the logarithmic Atiyah exact sequence (\S \ref{ssub:logarithmic_atiyah_exact_sequence}) is functorial in $\mathcal E$,	this follows from Lemma \ref{lem:sections-logarithmic-Atiyah-extension}.
\end{proof}

\begin{corollary}
\label{cor:sections-parabolic-Atiyah-extension}
There is a natural equivalence of categories between
\begin{itemize}
	\item the category $\Par\Con_{\frac{1}{\mathbf r}}(X,\mathbf D)$ of \preparabolic connections,
	\item the category $\Par\Sec_{\frac{1}{\mathbf r}}(X, \mathbf D)$, whose objects are couples $(\mathcal E_\cdot,\alpha_\cdot)$, where $\mathcal E_{\cdot}$ is a parabolic bundle and $\alpha_\cdot$ is a section of the parabolic Atiyah exact sequence of $\mathcal E_\cdot$, with obvious morphisms.
\end{itemize}
\end{corollary}

\subsection{Reconstruction of the parabolic structure}%
\label{sub:reconstruction_of_the_parabolic_structure}

It turns out that given a \parabolic connection \( (\mathcal{E}_\cdot, \nabla_\cdot) \), the underlying bundle  \( (\mathcal{E}_0, \nabla_0 )\), endowed with its logarithmic connection, enables to reconstruct the parabolic structure. To explain precisely how this is possible, one needs to consider parabolic bundles from a slightly different point of view.

\subsubsection{Seshadri's definition}%
\label{ssub:seshadri_s_definition}

\begin{definition}[]
\label{def:weight_filtration}
Let  \( \mathcal{E}_\cdot \) be a parabolic bundle with weights in  $\frac{1}{\mathbf r} \mathbb Z^I$. The \emph{weight filtration} on $ { \mathcal E_0} _{|D}$ is the filtration indexed by $\frac{1}{\mathbf r} \mathbb Z^I\cap [0,1]^I$ given by 

\[ F^w _{ \frac{ \mathbf{l}} { \mathbf{r}}}\left( { \mathcal E_0} _{|D} \right) = \frac{  \mathcal E_{ \frac{ \mathbf{l}} { \mathbf{r}}}}{\mathcal E_0(-D)} \] 

\end{definition}

One can introduce a category $\Sesh_{\frac{\mathbf 1}{\mathbf r}}(X,\mathbf D)$ whose objects are couples $(\mathcal{E},F_\cdot)$ where $\mathcal{E}$ is a vector bundle on $X$ and $F_\cdot$ is a decreasing filtration on $ { \mathcal E_0} _{|D}$ indexed by $\frac{1}{\mathbf r} \mathbb Z^I\cap [0,1]^I$ such that $F_{\mathbf 0} \mathcal E _{|D}= \mathcal E _{|D}$ and $F_{\mathbf 1} \mathcal E _{|D}=0$. It is clear that the functor  \( \mathcal{E}_\cdot  \mapsto ( {\mathcal E_0} _{|D},F^w _\cdot ) \) enables to see $\Par_{\frac{\mathbf 1}{\mathbf r}}(X,\mathbf D)$ as a full subcategory of $\Sesh_{\frac{\mathbf 1}{\mathbf r}}(X,\mathbf D)$.

\begin{definition}[]
\label{def:cartesian_filtration}
Let $\mathcal G$ be a sheaf on $D$, and $F_\cdot$ a decreasing filtration on $\mathcal G$ indexed by $\frac{1}{\mathbf r} \mathbb Z^I\cap [0,1]^I$ such that $F_{ \mathbf{0} } \mathcal G= \mathcal{G}$ and $F_{ \mathbf{1}} \mathcal G=0$. We will say that the filtration is cartesian if for all \( \frac{ \mathbf{l}} { \mathbf{r}} \)  in $\frac{1}{\mathbf r} \mathbb Z^I\cap [0,1]^I$ :

\[ F_{ \frac{ \mathbf{l}} { \mathbf{r}}}\left(\mathcal G \right) = \bigcap_{i\in I} F_{ \frac{l_i}{r_i}e_i}\left(\mathcal G \right) \]
where $(e_i)_{i\in I}$ stands for the canonical basis of $ \mathbb Z^I$.
\end{definition}

\begin{lemma}
\label{lem:weight_filtration_cartesian}

Let \( \mathcal{E}_\cdot \) be a parabolic bundle with weights in  $\frac{1}{\mathbf r} \mathbb Z^I$.  The weight filtration $F^w_\cdot$ on $ { \mathcal E_0} _{|D}$ is cartesian.
\end{lemma}

\begin{proof}
	This follows from Lemma \ref{lem:parabolic_bundle_restricted_axes}.
\end{proof}

\subsubsection{Connections with semi-simple residues}%
\label{ssub:connections_with_semi_simple_residues}

\begin{definition}[]
\label{def:connections_semi-simple_residues}
We denote by  $\Con_{\frac{\mathbf 1}{\mathbf r}}^{ss}(X,\mathbf D)$ the category whose objects are logarithmic connections $(\mathcal E,\nabla)$ along $D$ such that for each $i\in I$ the residue $\res_{D_i} \nabla$ is semi-simple with eigenvalues in  $\frac{1}{r_i} \mathbb Z\cap [0,1[$.
\end{definition}

For such a connection, one denotes for each $i\in I$ and each $\frac{l_i}{r_i} \in \frac{1}{r_i} \mathbb Z\cap [0,1[$
by  $\mathcal{E}_{|D_i}( \frac{l_i}{r_i})$ the subsheaf corresponding to the eigenvalue   $\frac{l_i}{r_i}$ with respect to the residue $\res_{D_i} \nabla$. We get a filtration $F^\nabla_\cdot$ indexed by $ \frac{1}{r_i} \mathbb{Z}\cap[0,1]$ 
$$F_{\frac{l_i}{r_i}}^\nabla( \mathcal{E}_{|D_i})=\oplus_{l_i\leq m_i < r_i}\mathcal{E}_{|D_i}( \frac{m_i}{r_i})  $$ 
of $\mathcal{E}_{|D_i}$ such that $F_{0}^\nabla( \mathcal{E}_{|D_i})=\mathcal{E}_{|D_i}$ and $F_{1}^\nabla( \mathcal{E}_{|D_i})=0$. By pulling-back along the canonical epimorphism $\mathcal E_{|D} \twoheadrightarrow \mathcal{E}_{|D_i}$ one gets a filtration $F^\nabla_\cdot$ of $\mathcal E_D$  indexed by $ \frac{1}{r_i} \mathbb{Z}\cap[0,1]$ such that $F_{0}^\nabla( \mathcal{E}_{|D})=\mathcal{E}_{|D}$ and $F_{1}^\nabla( \mathcal{E}_{|D})= \frac{\mathcal{E}(-D_i)}{\mathcal{E}(-D)}$. This filtration extends to a filtration $F^\nabla_\cdot$ of $\mathcal E_D$ indexed by   $\frac{1}{\mathbf r} \mathbb Z^I\cap [0,1]^I$ by putting

\[ F^\nabla_{ \frac{ \mathbf{l}} { \mathbf{r}}}\left(\mathcal E_D \right) = \bigcap_{i\in I} F^\nabla_{ \frac{l_i}{r_i}}\left(\mathcal E_D  \right) \; .\] 
This filtration is, by construction, cartesian, and verifies $F_{ \mathbf{0} }\left(\mathcal E_D \right)  = \mathcal E_D$ and $F_{ \mathbf{1}}\left(\mathcal E_D \right) =0$.

\subsubsection{Coincidence of the filtrations}%
\label{ssub:coincidence_of_the_filtrations}

\begin{proposition}
\label{prop:parabolic-connection-semisimple-residues}
Let \( (\mathcal{E}_\cdot, \nabla_\cdot) \) be a \parabolic connection on  $(X, \mathbf D)$ with weights in $\frac{1}{\mathbf r}\mathbb Z^I$. The logarithmic connection   \( (\mathcal{E}_0, \nabla_0 )\) has semi-simple residues with eigenvalues in  $\frac{1}{\mathbf r} \mathbb Z^I\cap [0,1[^I$ and moreover the filtration of ${\mathcal E_0}_{|D}$ associated to $\nabla_0$ coincides with the weight filtration associated to the parabolic bundle \( \mathcal{E}_\cdot \), that is $F^w_\cdot=F^{\nabla_0}_\cdot$. 
\end{proposition}

In other words, in the commutative diagram

\begin{center}
	\begin{tikzcd}
                     &	(\mathcal E_\cdot,\nabla_\cdot) \arrow[r, maps to]              & \mathcal E_\cdot          &                      \\
		(\mathcal E_\cdot,\nabla_\cdot) \arrow[d, maps to] &  \Par\Con_{\frac{\mathbf 1}{\mathbf r}}^{st}(X,\mathbf D) \arrow[d] \arrow[r]             & \Par_{\frac{\mathbf 1}{\mathbf r}}(X,\mathbf D) \arrow{d} & \mathcal E_\cdot \arrow[d, maps to] \\
		(\mathcal E_{ \mathbf{0}},\nabla_{ \mathbf{0}})                   & \Con^{ss}_{\frac{\mathbf 1}{\mathbf r}}(X,\mathbf D) \arrow[r]                       & \Sesh_{\frac{\mathbf 1}{\mathbf r}}(X,\mathbf D)          & (\mathcal E_{ \mathbf{0}},F^w_\cdot)                   \\
							  & ( \mathcal{E},\nabla) \arrow[r, maps to, shift right] & ( \mathcal{E},F_\cdot^\nabla)           &                     
\end{tikzcd}
\end{center}
the left hand functor is well defined and the diagram commutes. Before proving Proposition \ref{prop:parabolic-connection-semisimple-residues} we state a consequence:

\begin{corollary}
\label{cor:parabolic-connection-semisimple-residues}
Let \( (\mathcal{E}_\cdot, \nabla_\cdot) \) and \( (\mathcal{E'}_\cdot, \nabla'_\cdot) \) be two \parabolic connections on  $(X, \mathbf D)$ with weights in $\frac{1}{\mathbf r}\mathbb Z^I$. Then any isomorphism $(\mathcal E_{ \mathbf{0}},\nabla_{ \mathbf{0}})  \simeq (\mathcal E'_{ \mathbf{0}},\nabla'_{ \mathbf{0}})  $ lifts uniquely to an isomorphism \( (\mathcal{E}_\cdot, \nabla_\cdot) \simeq (\mathcal{E'}_\cdot, \nabla'_\cdot) \)
\end{corollary}

\begin{proof}
If $(\mathcal E_{ \mathbf{0}},\nabla_{ \mathbf{0}})  \simeq (\mathcal E'_{ \mathbf{0}},\nabla'_{ \mathbf{0}})  $, then  Proposition \ref{prop:parabolic-connection-semisimple-residues} implies that the isomorphism $\mathcal E_{ \mathbf{0}} \simeq \mathcal E'_{ \mathbf{0}}  $ lifts to an isomorphism of parabolic bundles $\mathcal E_{ \cdot} \simeq \mathcal E'_{ \cdot}  $. But since $\mathcal E_{ \mathbf{0}} \simeq \mathcal E'_{ \mathbf{0}}  $ is compatible with the connections, the pseudo-periodicity isomorphisms ensure that $\mathcal E_{ \cdot} \simeq \mathcal E'_{ \cdot}  $ is compatible with the connections as well.
\end{proof}

In order to prove Proposition \ref{prop:parabolic-connection-semisimple-residues}, we first recall a well-known lemma.

\begin{lemma}
\label{lem:locally-free-cokernel}
Let $Y$ be a scheme, $\phi: \mathcal{E}\rightarrow \mathcal{E'}$ a morphism of finite locally free sheaves, and for each $y\in Y$, denote by $\rk_y \phi$ the rank of the $k(y)$-linear morphism $\phi\otimes k(y): \mathcal{E}\otimes_{\mathcal{O}_Y} k(y)\rightarrow \mathcal{E}'\otimes_{\mathcal{O}_Y} k(y)$. Then:
\begin{enumerate}
	\item if $\coker \phi$ is locally free, then $y\mapsto  \rk_y \phi$ is locally constant,
	\item if $Y$ is reduced and $y\mapsto  \rk_y \phi$ is locally constant then $\coker \phi$ is locally free.
\end{enumerate}

\end{lemma}

\begin{proof}
	Considering the fact that the functor $\cdot\otimes_{\mathcal{O}_Y} k(y)$ is right exact, the first point is clear, and the second point follows from Nakayama's lemma, see for instance \cite[\href{https://stacks.math.columbia.edu/tag/0FWH}{Tag 0FWH}]{stacks-project}.
\end{proof}

\begin{remark}
\label{rem:locally-free-cokernel}
If $\phi: \mathcal{E}\rightarrow \mathcal{E'}$ a morphism of finite locally free sheaves with locally free cokernel then $\im \phi$ and $\ker \phi$ are also locally free. Indeed if $0 \rightarrow \mathcal E''\rightarrow \mathcal{E}\rightarrow \mathcal{E'} \rightarrow 0$ is an exact sequence of quasi-coherent sheaves and $\mathcal E$ and  $\mathcal{E'}$ are finite locally free, the exact sequence splits locally, hence $\mathcal E''$ is finite locally free as well. \end{remark}

\begin{definition}[]
\label{def:subbundle}
If $Y$ is a scheme, $\phi: \mathcal{E}\hookrightarrow \mathcal{E}'$ is a monomorphism of locally free sheaves, we will say that  $\mathcal{E}$ is a \emph{subbundle} of $\mathcal E'$ if $\coker \phi$ is a locally free sheaf.
\end{definition}

\begin{remark}
\label{rem:subbundle}
As in Remark \ref{rem:locally-free-cokernel}, if $\mathcal{E}''\subset\mathcal{E}'\subset \mathcal E$ and $\mathcal{E}'$ is a subbundle of $\mathcal{E}$, then $\mathcal{E}''$ is a subbundle of $\mathcal{E}$ if and only if $\mathcal{E}''$ is a subbundle of $\mathcal{E'}$.
\end{remark}

\begin{lemma}
\label{lem:filtration-of-semisimple}
Let $Y$ be reduced scheme over a field $k$, and $\phi$ be an endomorphism of a finite locally free sheaf  $\mathcal{E}$. Assume that $\mathcal E$ admits a filtration by $\phi$-stable subbundles 
$F_0\mathcal{E} =\mathcal{E}\supset F_1\mathcal{E}\supset \cdots \supset F_N\mathcal{E} \supset F_{N+1}\mathcal{E}=\{0\}$ such that $\phi$ induces $\lambda_i \id$ on $\frac{F_i\mathcal{E}}{F_{i+1}\mathcal{E}}$ for $i=0,\cdots,N$, where the $\lambda_i \in k$ are pairwise distinct. Then:
\begin{enumerate}
	\item $\phi$ is semisimple, that is $\mathcal E= \oplus_{i=0}^N \mathcal E(\lambda_i)$, where $\mathcal E(\lambda_i)$, the eigen-subsheaf associated to $\lambda_i$, is locally free of rank $\rk \frac{F_i\mathcal{E}}{F_{i+1}\mathcal{E}}	$,
	\item for $i=0,\cdots,N$ it holds moreover that $F_i\mathcal{E}=\oplus_{j=i}^N\mathcal E(\lambda_j)$.
\end{enumerate}
\end{lemma}

\begin{proof}
	One first observes that the second assertion follows from the first, namely if the first point is true then $F_1 \mathcal E(\lambda_i)=\mathcal E(\lambda_i)$ for $i\geq 1$, hence the second point follows by induction on the length $N$ of the filtration.
	
	To prove the first assertion, one notes that is it valid over the spectrum of a field, as $P=(t-\lambda_0)\cdots(t-\lambda_N)$ is a polynomial such that $P(\phi)=0$. From this, it follows that for $i=0,\cdots,N$, the morphism $\phi-\lambda_i\id $ is of locally constant rank $\rk \mathcal E - \rk \frac{F_i\mathcal{E}}{F_{i+1}\mathcal{E}}	$, hence according Remark \ref{rem:locally-free-cokernel} $\mathcal E(\lambda_i)$ is a subbundle of $\mathcal E$ of rank $ \rk \frac{F_i\mathcal{E}}{F_{i+1}\mathcal{E}}$. One concludes that the morphism $\oplus_{i=0}^N \mathcal E(\lambda_i) \rightarrow \mathcal{E} $ is of locally constant rank $\rk \mathcal{E}$, hence, according to Lemma \ref{lem:locally-free-cokernel}, it is an isomorphism.

\end{proof}
\begin{proof}[Proof of Proposition \ref{prop:parabolic-connection-semisimple-residues}]
Lemma \ref{lem:weight_filtration_cartesian} allows to reduce to the case of a single divisor, that is $\# I=1$, so we can omit indices. Now \cite[Lemme 2.3.11]{bor:rep} shows that $ {\mathcal E_0}_{|D}\supset \frac{\mathcal E_ \frac{1}{r}  }{\mathcal E_0(-D)} \supset \cdots\frac{\mathcal E_ \frac{r-1}{r}  }{\mathcal E_0(-D)}  \supset 0 $ is a filtration by $\mathcal{O}_D$-subbundles, hence Lemma \ref{lem:filtration-of-semisimple} applies, which concludes the proof.
 
\end{proof}

\begin{remark}
\label{rem:parabolic-connection-semisimple-residues}
In view of Proposition \ref{prop:parabolic-connection-semisimple-residues}, it is natural to ask if the functor $\Par\Con_{\frac{\mathbf 1}{\mathbf r}}^{st}(X,\mathbf D) \rightarrow  \Con^{ss}_{\frac{\mathbf 1}{\mathbf r}}(X,\mathbf D)  $ given on objects by $(\mathcal E_\cdot,\nabla_\cdot)\mapsto (\mathcal E_0,\nabla_0) $ is essentially surjective. The answer is negative: namely if one starts from an object $(\mathcal E,\nabla)$ in $\Con^{ss}_{\frac{\mathbf 1}{\mathbf r}}(X,\mathbf D)$, the given condition on the residues is too weak to ensure that the filtration $F^\nabla_\cdot$ is stable by $\nabla$. For an explicit counter-example, one can consider $X= \mathbb{A}^2$, $D=\{0\}\times \mathbb{A}^1$, $r=2$, $\mathcal{E}=\mathcal O_X^{\oplus 2}$, and $\nabla$ given by the matrix of $1$-forms $\Omega= 
\begin{pmatrix}
\frac{1}{2} 	\frac{dx}{x} & 0\\
	dy & 0
\end{pmatrix}$.
Then one easily checks that the corresponding parabolic vector bundle verifies $\mathcal{E}_{ \frac{1}{2}}= \mathcal{O}_X\oplus  \mathcal{O}_X(-D)$, and as $dy\notin \Omega^1_X(\log D)(-D) $, this submodule is not stable by $\nabla$. 
\end{remark}

\subsection{From stacky connections to parabolic connections}%
\label{sub:from_stacky_connections_to_parabolic_connections}

\subsubsection{From stacky logarithmic connections to \preparabolic connections}%
\label{ssub:from_stacky_logarithmic_connections_to_pre_parabolic_connections}

\begin{definition}[]
\label{def:functor_F_enriched}
To each logarithmic connection $(\mathcal F,\nabla)$ on $(\mathfrak X, \mathbf{ \mathfrak{D} })$, one associates a \preparabolic connection $(\widehat{\mathcal F}_\cdot,\widehat{\nabla}_\cdot)$  on $(X,\mathbf D)$ with weights in  $\frac{1}{\mathbf r}\mathbb Z^I$ in the following way:

\begin{itemize}
	\item the underlying parabolic vector bundle $\widehat{\mathcal F}_\cdot$ is the one associated to $ \mathcal{F}$ by Definition \ref{def:functors_F_G},
	\item  if $\mathbf l$ belongs to 
 $\mathbb Z^I$, one defines $\widehat{\nabla}_{\frac{\mathbf l}{\mathbf r}} =\pi_* 
 \left( \nabla ( -\mathbf  l \mathbf{\mathfrak D  })\right)$.

\end{itemize}
\end{definition}

\begin{remark}
\label{rem:functor_F_enriched}

\begin{enumerate}
	\item To define the second part of the data, we have used the natural extension to Deligne-Mumford stacks of the two operations on logarithmic connections met previously for schemes:

\begin{itemize}
	\item the twist of a logarithmic connection on $(\mathfrak X, \mathbf{ \mathfrak{D} })$ by a divisor with support in    $\mathfrak{D}$ (see \ref{lem:canonical_connection}),
	\item the push-forward of a logarithmic connection on $(\mathfrak X, \mathbf{ \mathfrak{D} })$ along the log-étale morphism  $(\pi,(r_i)_{i\in I}): (\mathfrak{X},( \mathfrak{D}_i)_{i\in I}) \rightarrow (X,(D_i)_{i\in I})$ (see %Definition \ref{def:log-etale}
		Remark \ref{rem:push-forward-logarithmic-connection}). 
\end{itemize}

\item The fact that the $\widehat{\nabla}_{\frac{\mathbf l}{\mathbf r}}$ are compatible with themselves is a consequence of Remark \ref{rem:canonical connection}.

\end{enumerate}

\end{remark}

Our next task is to show the following:
\begin{theorem}[]
\label{thm:car_hol_stacky}
A logarithmic connection on $(\mathcal F,\nabla)$ on $(\mathfrak X, \mathbf{ \mathfrak{D} })$ is holomorphic if and only if the associated  \preparabolic connection $(\widehat{\mathcal F}_\cdot,\widehat{\nabla}_\cdot)$ is \parabolic.

\end{theorem}

The proof is postponed until \S \ref{ssub:proof_of_the_only_if_direction_of_theorem_thm_car_hol_stacky} and \S\ref{ssub:proof_of_the_if_direction_of_theorem_thm_car_hol_stacky}.

\subsubsection{Residues and push-forward}%
\label{ssub:residues_and_push_forward}

We now select an index $h\in I$ \footnote{The letter is chosen in order to avoid confusion with the closed immersions.}. There is a canonical commutative diagram: 

\begin{center}
\begin{tikzcd}
\mathfrak D_h \arrow[r, "j_h"] \arrow[d, "p_h"'] & \mathfrak X \arrow[d, "\pi"] \\
D_h \arrow[r, "i_h"']                            & X                           
\end{tikzcd}
\end{center}

\begin{lemma}
\label{lem:exact_sequence_of_restrictions}
Let $\mathcal F$ be a vector bundle on $\mathfrak X$. There is a natural exact sequence:

$$  i_h^*\pi_*(\mathcal F\otimes_{\mathcal O_{ \mathfrak{X} }} \mathcal O_{ \mathfrak{X} }(-\mathfrak D_h)) \rightarrow i_h^*\pi_*\mathcal F\xrightarrow{c_h}  {p_h}_*j_h^*\mathcal F \rightarrow 0$$
\end{lemma}

\begin{proof}

	Since $ \mathcal{F}$ is locally free, by the projection formula, ${j_h}_*j_h^*\mathcal F \simeq \mathcal{F}\otimes_{ \mathcal{O}_{ \mathfrak{X} } } {j_h}_* \mathcal{O}_{ \mathfrak{D}_h}  $, so there is an exact sequence:

	\[ 0 \rightarrow \mathcal{F}\otimes_{ \mathcal{O}_{ \mathfrak{X} } } \mathcal{O}_{ \mathfrak{X} }(- \mathfrak{D}_h) \rightarrow \mathcal{F} \rightarrow {j_h}_*j_h^*\mathcal F \rightarrow 0 \]
	As the stack of roots is tame (see \cite{AOV:tame}), the functor $\pi_*$ is exact, hence we get a natural exact sequence:

\[ i_h^*\pi_* \left( \mathcal{F}\otimes_{ \mathcal{O}_{ \mathfrak{X} } } \mathcal{O}_{ \mathfrak{X} }(- \mathfrak{D}_h)\right) \rightarrow i_h^*\pi_* \mathcal{F} \rightarrow  i_h^*\pi_*{j_h}_*j_h^*\mathcal F \rightarrow 0 \]
But now $i_h^*\pi_*{j_h}_*j_h^*\mathcal F\simeq i_h^* {i_h}_* {p_h}_* j_h^*\mathcal F \simeq {p_h}_* j_h^*\mathcal F$.
\end{proof}

\begin{lemma}
\label{lem:morphism_restrictions_compatible_residues}
Let $(\mathcal F,\nabla)$ be a logarithmic connection on $(\mathfrak X,\mathfrak D)$. The canonical epimorphism 

$$c_h: i_h^*\pi_*\mathcal F\rightarrow  {p_h}_*j_h^*\mathcal F $$
is compatible with the endomorphisms  $\operatorname{res}_{D_h} (\pi_*\nabla)$ on $i_h^*\pi_*\mathcal F$ and $\frac{1}{r_h} p_* \operatorname{res}_{\mathfrak D_h}(\nabla)$ on ${p_h}_*j_h^*\mathcal F$, in other words:

$$  c_h\circ \operatorname{res}_{D_h} (\pi_*\nabla)= \left(\frac{1}{r_h} p_* \operatorname{res}_{\mathfrak D_h}(\nabla) \right)\circ c_h \; .$$

\end{lemma}

\begin{proof}

This is equivalent to the commutativity of the natural diagram:
	\begin{center}
		\begin{tikzcd}
			\pi_*\mathcal F \arrow[r, "\pi_*\nabla"]                   & \pi_*(\mathcal F\otimes_{\mathcal O_{\mathfrak X}} \Omega_\mathfrak X^1(\log(\mathfrak D))) \arrow[r, "\pi_*(\id\otimes\res)"] & \pi_*{j_h}_*j_h^*\mathcal F={i_h}_*{p_h}_*j_h^*\mathcal F \\
\pi_*\mathcal F \arrow[r, "\pi_*\nabla"'] \arrow[u, "\id"] & \pi_*\mathcal F \otimes_{\mathcal O_X} \Omega^1_X(\log(D)) \isoarrow{u} \arrow[r, "\id\otimes\res"']                        & {i_h}_*i_h^*\pi_*\mathcal F  \arrow[u, "\times r_h"']  
\end{tikzcd}
	\end{center}
	The left-hand square commutes by definition of $\pi_*\nabla$ (Remark \ref{rem:push-forward-logarithmic-connection}). The right-hand square being $\mathcal O_X$-linear and independent of $\nabla$, the commutativity is also easily checked. Namely, by adjunction, it is equivalent to the commutativity of
\begin{center}
\begin{tikzcd}
\mathcal F\otimes_{\mathcal O_{\mathfrak X}} \Omega_\mathfrak X^1(\log(\mathfrak D)) \arrow[r]        & {j_h}_*j_h^*\mathcal F  \arrow[r, "\sim"]                   & \mathcal F \otimes_{\mathcal O_{\mathfrak X}} {j_h}_*\mathcal O_{\mathfrak D_h} \\
\pi^*\pi_*\mathcal F \otimes_{\mathcal O_{\mathfrak X}} \pi^*\Omega^1_X(\log(D))  \arrow[u] \arrow[r] & \pi^*{i_h}_*i_h^*\pi_*\mathcal F \arrow[u, "\times r_h"'] \arrow[r, "\sim"] & \pi^*\pi_*\mathcal F\otimes_{\mathcal O_{\mathfrak X}} \pi^*{i_h}_*\mathcal O_{D_h}  \arrow[u, "\times r_h"']                                                
\end{tikzcd}                                   
\end{center}
so we are reduced to the case where \( \mathcal{F} =  \mathcal{O}_{\mathfrak X} \). But the result is then clear, since $r_h \mathfrak D_h=\pi^* D_h$: if $t_h$ is a local equation of $\mathfrak D_h$, then $s_h=t_h^{r_h}$ is a local equation of $\pi^* D_h$, and $ \frac{ds_h}{s_h} =	r_h \frac{dt_h}{t_h}$, which proves the commutativity.

\end{proof}

\subsubsection{Proof of the `only if' direction of Theorem \ref{thm:car_hol_stacky}}%
\label{ssub:proof_of_the_only_if_direction_of_theorem_thm_car_hol_stacky}

Let  $(\mathcal F,\nabla)$ a holomorphic connection on $(\mathfrak X, \mathbf{ \mathfrak{D} })$, and as usual let 
 $(\widehat{\mathcal F}_\cdot,\widehat{\nabla}_\cdot)$ be the associated  \preparabolic connection.
Let $\mathbf l\in \mathbb Z^I$. We apply Lemma \ref{lem:morphism_restrictions_compatible_residues} to the logarithmic connection $(\mathcal F( -\mathbf  l \mathbf{\mathfrak D  }),\nabla( -\mathbf  l \mathbf{\mathfrak D  }))$.
It shows that the morphism $c_h$ identifies with

\[ i_h^* \widehat{\mathcal F}_{\frac{\mathbf l}{\mathbf r}}\twoheadrightarrow i_h^*\left( \frac{ \widehat{\mathcal F}_{\frac{\mathbf l}{\mathbf r}}}{ \widehat{\mathcal F}_{\frac{\mathbf l+e_h}{\mathbf r}}} \right) \]
So the endomorphism $\operatorname{res}_{D_h} (\widehat{\nabla}_{\frac{\mathbf l}{\mathbf r}})=\operatorname{res}_{D_h} (\pi_*\nabla( -\mathbf  l \mathbf{\mathfrak D  }))$ of the left hand side induces the endomorphism $\frac{1}{r_h} p_* \operatorname{res}_{\mathfrak D_h}(\nabla( -\mathbf  l \mathbf{\mathfrak D  }))$ on the right hand side. But as 
$(\mathcal F,\nabla)$ is holomorphic, Definition \ref{lem:canonical_connection} shows that $\operatorname{res}_{\mathfrak D_h}(\nabla( -\mathbf  l \mathbf{\mathfrak D  }))= l_h \id$. Hence the endomorphism induced by $\operatorname{res}_{D_h} (\widehat{\nabla}_{\frac{\mathbf l}{\mathbf r}})$ is $\frac{l_h}{r_h}\id$. Since this is true for any $\mathbf l\in \mathbb Z^I$ and any $h\in I$, the \preparabolic connection $(\widehat{\mathcal F}_\cdot,\widehat{\nabla}_\cdot)$ is in fact \parabolic.

\subsubsection{Proof of the `if' direction of Theorem \ref{thm:car_hol_stacky}}%
\label{ssub:proof_of_the_if_direction_of_theorem_thm_car_hol_stacky} 

Let  $(\mathcal F,\nabla)$ a logarithmic connection on $(\mathfrak X, \mathbf{ \mathfrak{D} })$, we assume that the associated  \preparabolic connection $(\widehat{\mathcal F}_\cdot,\widehat{\nabla}_\cdot)$ is a \parabolic connection. We have to show that $\res_{\mathfrak D_h} \nabla = 0$ for all $h\in I$. 

We first reduce to the case where $\# I=1$. To do so, we observe more generally that if $\mathfrak D$ is a Cartier divisor on a Deligne-Mumford stack $\mathfrak X$, its complement $\mathfrak U =\mathfrak X \backslash \mathfrak D$ is scheme-theoretically dense (\cite[\href{https://stacks.math.columbia.edu/tag/01RE}{Tag 01RE}]{stacks-project}), that is, if $\iota:\mathfrak U \rightarrow \mathfrak{X}$ is the open immersion, then $ \mathcal{O}_{\mathfrak X} \hookrightarrow {\iota}_* \mathcal{O}_{\mathfrak U}$. From the projection formula, we get that for any locally free sheaf $\mathcal F$ on $ \mathfrak{X}$, the morphism $ \mathcal{F} \rightarrow \iota_*\iota^*  \mathcal{F}$ is injective.
%Hence to show that a section of a locally free sheaf is $0$, we can check it on $\mathfrak U$.

It follows that to prove that $\res_{\mathfrak D_h} \nabla = 0$ as an endomorphism of $ \mathcal{F}_{ |\mathfrak{D}_h}$, one can check it on $\mathfrak D_h \backslash \left(\cup _{h'\in I, h'\neq h} \mathfrak D_h'\cap \mathfrak D_h \right)$ (it follows from the proof of \cite[\href{https://stacks.math.columbia.edu/tag/00NQ}{Tag 00NQ}]{stacks-project} that $\left(\cup _{h'\in I, h'\neq h} \mathfrak D_h'\cap \mathfrak D_h \right)$ is a sncd in $\mathfrak D_h$). 
In order to do so, we can work on $\mathfrak X \backslash \left(\cup _{h'\in I, h'\neq h} \mathfrak D_h' \right)$ (where $ \mathfrak{X}$ stands again for our stack of roots), hence we can assume that $\# I=1$.

So we can forget about indices, and we will thus work with the following notations.

\begin{center}
\begin{tikzcd}
\mathfrak D \arrow[r, "j"] \arrow[d, "p"'] & \mathfrak X \arrow[d, "\pi"] \\
D \arrow[r, "i"']                            & X                           
\end{tikzcd}
\end{center}
It is clear that $\mathfrak D \rightarrow D$ is a $\mu_r$-gerbe, but we can actually say a bit more.

\begin{definition}[]
\label{def:gerbe_of_roots}

If $S$ a scheme, and $\mathcal L$ is an invertible sheaf on $S$, we denote by $\sqrt[r]{\mathcal L/S}$ the \emph{gerbe of $r$-th roots of $ \mathcal{L}$}, that is, the gerbe whose objects over $f:S'\to S$ are invertible sheaves $\mathcal M$ on $S'$ endowed with an isomorphism $f^* \mathcal{L}\simeq \mathcal{M}^{\otimes r}$. 
\end{definition}

\begin{lemma}
\label{lem:natural_gerbe_of_roots}

Let $X$ be a scheme, $D$ an effective Cartier divisor on $X$, $r\geq 1$ an integer, and $ \mathfrak{X}=\sqrt[r]{D/X}$. We denote by $ \mathfrak{D}$ the canonical Cartier divisor on $ \mathfrak{X}$ and by $ \mathcal{N}_D= \mathcal{O}_X(D)_{|D}$ and $ \mathcal{N}_{ \mathfrak{D}}= \mathcal{O}_X( \mathfrak{D})_{| \mathfrak{D} }$ the conormal sheaves. Then there is a canonical $D$-isomorphism $ \mathfrak{D}\simeq \sqrt[r]{\mathcal N_D/D}$ sending the canonical $r$-th root of $ \mathcal{N}_D$ to $ \mathcal{N}_{ \mathfrak{D}}$.

\end{lemma}

\begin{proof}

	First note that the pull-back of the closed immersion $B \mathbb{G}_m \hookrightarrow [\mathbb A^1| \mathbb G_m]$  defined by $\mathcal L \mapsto ( \mathcal{L},0)$ by $(\mathcal O_X(D),s_D): X \rightarrow [\mathbb A^1| \mathbb G_m]$ is just $i:D \hookrightarrow X$, in other words in the following natural commutative diagram 

	\begin{center}
		\begin{tikzcd}
                                                                                                         & \mathfrak D \arrow[dd] \arrow[ld] \arrow[rr] &                                                     & B\mathbb G_m \arrow[dd] \arrow[ld] \\
\mathfrak X \arrow[dd] \arrow[rr] &                                                          & {[\mathbb A^1| \mathbb G_m]} \arrow[dd] &                                                \\
                                                                                                         & D \arrow[rr] \arrow[ld]                            &                                                     & B\mathbb G_m \arrow[ld]                        \\
X \arrow[rr]                                                                 &                                                          & {[\mathbb A^1| \mathbb G_m]}                        &                                               
\end{tikzcd}

	\end{center}
 the bottom face is cartesian. For the same reason the top face is also cartesian, and since the front face is cartesian by Definition \ref{def:stack_of_roots}, the back face is cartesian as well.

\end{proof}

This description of \(p: \mathcal{D} \rightarrow D \) is useful as the representation theory of \( \mathbb{\mu}_r \)-gerbes associated to an invertible sheaf is fairly simple:

\begin{lemma}
\label{lem:representation_theory_mu_r_gerbe}
Let $S$ be a scheme, $\mathcal L$ be an invertible sheaf on $S$, $r\geq 1$ be an integer, and  $p: \mathfrak G=\sqrt[r]{\mathcal L/S} \rightarrow S$ be the associated  \( \mathbb{\mu}_r \)-gerbe. We denote by $ \mathcal{N}$ the canonical $r$-th root of $\mathcal L$ on \( \mathfrak{G} \). If $ \mathcal{G}$ is a locally free sheaf on  \( \mathfrak{G} \), then the canonical morphism:

\[  \bigoplus_{l=0}^{r-1} p^*p_* \left( \mathcal G \otimes_{ \mathcal{O}_{ \mathfrak{G} } } \mathcal N^{\vee\otimes l} \right)\otimes_{ \mathcal{O}_{ \mathfrak{G} } }  \mathcal N^{\otimes l} \rightarrow \mathcal{G}\]
is an isomorphism.
\end{lemma}

\begin{proof}
	This property is Zariski-local on $S$, so we can assume $S= \Spec(R)$ and that $\mathcal L$ admits a $r$-th root on $S$, in other words there is a $S$-isomorphism $B_S   \mathbb{\mu}_r \simeq \mathfrak{G}$. But now the category of quasi-coherent sheaves on \( \mathfrak{G} \) is equivalent to the category of \( \frac{ \mathbb{Z}}{r}  \)-graded $R$-modules, and the isomorphism boils down to the following obvious isomorphism 
	\[ \bigoplus_{l=0}^{r-1} \left(M \otimes_{R} R[-l]\right)_0 \otimes_{R} R[l] \simeq M \; .\]
	Namely, as we assume that $M$ is locally free, we have that $M \otimes_{R} R[-l] \simeq \Hom(R[l],M)=M_{-l}$, hence the result.
\end{proof}

\begin{corollary}
\label{cor:representation_theory_mu_r_gerbe}
With notations of Lemma \ref{lem:representation_theory_mu_r_gerbe}, if $\phi: \mathcal{G} \rightarrow \mathcal{G'}$ is a morphism between two locally free sheaves on \( \mathfrak{G} \), then \( \phi=0 \) if and only if 

\[ p_* \left( \phi \otimes \id_{\mathcal N^{\vee\otimes l}} \right) =0 \]
for $l\in \{0,\cdots, r-1\}$.
\end{corollary}

\begin{proof}
This is a direct consequence of Lemma \ref{lem:representation_theory_mu_r_gerbe}.
\end{proof}

We can now end the proof of Theorem \ref{thm:car_hol_stacky}.

Let  $(\mathcal F,\nabla)$ a logarithmic connection on $(\mathfrak X, \mathbf{ \mathfrak{D} })$, we assume that the associated  \preparabolic connection is a \parabolic connection. As we have seen in \S \ref{ssub:proof_of_the_only_if_direction_of_theorem_thm_car_hol_stacky}, this assumption implies that, for each $l\in \mathbb{Z}$, the residue $\res_D(\pi_*(\nabla (-l \mathfrak{D}))$ induces $ \frac{l}{r}\id $ on the right-hand side of the epimorphism

\[ c: i^*\pi_*\mathcal F(-l \mathfrak{D}) \rightarrow  p_*j^*\mathcal F(-l \mathfrak{D})
\; .	\]
From this and Lemma \ref{lem:morphism_restrictions_compatible_residues}, it follows that 
$ p_* \left(\res_{ \mathfrak{D}} (\nabla (-l \mathfrak{D}))\right) = l\id $.
But since  $\res_{ \mathfrak{D}} (\nabla (-l \mathfrak{D}))= \res_{ \mathfrak{D}} (\nabla)\otimes \id_{\mathcal N_D^{\vee\otimes l}} +l\id $, 
we get that $p_*\left( \res_{ \mathfrak{D}} (\nabla)\otimes \id_{\mathcal N_D^{\vee\otimes l}} \right)=0$. Since this is true for any $l\in \mathbb{Z}$, Lemma \ref{lem:natural_gerbe_of_roots} and Corollary \ref{cor:representation_theory_mu_r_gerbe} enable to conclude that $\res_{\mathfrak D} \nabla = 0$.

\subsection{From parabolic connections to stacky connections}%
\label{sub:from_parabolic_connections_to_stacky_connections}

\subsubsection{From \preparabolic connections to logarithmic stacky connections}%
\label{ssub:from_pre_parabolic_connections_to_logarithmic_stacky_connections}

\begin{lemma}
\label{lem:connection_on_colimit}
Let \( \mathfrak{X} \) be a Deligne-Mumford stack over \( k \), and let \( \mathfrak{D} \) be a sncd divisor. Let 
\( (\mathcal{F}_\cdot,\nabla_\cdot): J \rightarrow \Con( \mathfrak{X},\mathfrak{D})  \) be a diagram such that the colimit \( \mathcal{F}=\varinjlim_J \mathcal{F}_j \) exists in \( \Vect ( \mathfrak{X}) \).  Then there exists a unique logarithmic connection $\nabla$ on $\mathcal{F}$ such that for each $j \in J$ the morphism $ \mathcal{F}_j \rightarrow \mathcal{F}$ is compatible with $\nabla_j$ and $\nabla$.

\end{lemma}

\begin{proof}

	This follows from (the stacky version of) Corollary \ref{cor:sections-logarithmic-Atiyah-extension}. Namely as  $\varinjlim_J \mathcal F_j\otimes_{\mathcal O_{\mathfrak X}}\Omega_{\mathfrak X/k}^1(\log(\mathfrak D))\simeq \mathcal F\otimes_{\mathcal O_{\mathfrak X}}\Omega_{\mathfrak X/k}^1(\log(\mathfrak D))$, the natural morphism  $$\varinjlim_J P^1_{(X,D)/k}( \mathcal{F}_j) \rightarrow P^1_{(X,D)/k}(\mathcal F)$$ is an isomorphism as well.
\end{proof}

\begin{remark}
\label{rem:connection_on_colimit}
As coends are a special type of colimits (see \cite[IX \S5 Proposition 1]{MCL:categories}), Lemma \ref{lem:connection_on_colimit} holds if we change \( (\mathcal{F}_\cdot,\nabla_\cdot): J \rightarrow \Con( \mathfrak{X},\mathfrak{D})  \) by a functor of mixed variance \( (\mathcal{F}_{\cdot,\cdot} ,\nabla_{\cdot,\cdot}): J^{op}\times J \rightarrow \Con( \mathfrak{X},\mathfrak{D}) \) and the colimit by the coend  \( \mathcal{F}=\int^J \mathcal{F}_{j,j} \).
\end{remark}

To a \preparabolic connection $(\mathcal E_\cdot,\nabla_\cdot)$  is associated a functor of mixed variance:

\begin{center}
	\begin{tikzcd}
	\left(\frac{1}{\mathbf r}\mathbb Z^I\right)^{op}\times \frac{1}{\mathbf r}\mathbb Z^I \arrow[rr]       &  & \Con( \mathfrak{X},\mathfrak{D})  \\
	( \frac{ \mathbf{l} }{ \mathbf{r}}  , \frac{ \mathbf{l'} }{ \mathbf{r'}} )   	\arrow[rr, maps to] &  & \left(\pi^*\mathcal E_{\frac{ \mathbf{l} }{ \mathbf{r}} } \otimes \mathcal O_{\mathfrak X}( \mathbf{l'} \mathbf{\mathfrak D )}, \pi^* \nabla_{\frac{ \mathbf{l} }{ \mathbf{r}} }( \mathbf{l'} \mathbf{\mathfrak D} )  \right) \; .
\end{tikzcd}
\end{center}
Now Definition \ref{def:functors_F_G} and  Remark \ref{rem:connection_on_colimit} show that the following definition makes sense: 

\begin{definition}[]
\label{def:connection_on_coend}
Let $(\mathcal E_\cdot,\nabla_\cdot)$ be an object in $\Par\Con_{\frac{1}{\mathbf r}}(X,\mathbf D)$. We will denote by $\widehat{\nabla_\cdot}$ the unique connection on the vector bundle $\widehat{\mathcal E_\cdot}=\int^{\frac{1}{\mathbf r}\mathbb Z^I}\pi^*\mathcal E_\cdot \otimes \mathcal O_{\mathfrak X}( \cdot  \mathbf r \mathbf{\mathfrak D ) } $
compatible with the given connections on each of the term of the coend.
\end{definition}

\subsubsection{The tensor equivalence}%
\label{ssub:the_tensor_equivalence}

We first explain how to endow  $\Par\Con_{\frac{1}{\mathbf r}}(X,\mathbf D)$ with a natural tensor product.
In \cite[\S 2.1.3]{bor:rep}, the first author described the tensor product on category $\Par_{\frac{1}{\mathbf r}}(X,\mathbf D)$ as given by the convolution formula:

\[ \left( \mathcal{E}_\cdot \otimes \mathcal{E}'_\cdot \right)_{\frac{ \mathbf{l}}{ \mathbf{r}} }= 
\int^{\frac{ \mathbf{m}}{ \mathbf{r}}\in \frac{1}{\mathbf r} \mathbb{Z}^I} \mathcal{E}_{\frac{ \mathbf{m}}{ \mathbf{r}} }  \otimes \mathcal{E}'_{\frac{ \mathbf{l-m}}{ \mathbf{r}} }  \] 
If we start from two \preparabolic connections $(\mathcal{E}_\cdot,\nabla_\cdot)$ and $(\mathcal{E}'_\cdot,\nabla'_\cdot)$, each term $\mathcal{E}_{\frac{ \mathbf{m}}{ \mathbf{r}} }  \otimes \mathcal{E}'_{\frac{ \mathbf{l-m}}{ \mathbf{r}} }$ is endowed with a tensor product logarithmic connection $\nabla_{\frac{ \mathbf{m}}{ \mathbf{r}} }  \otimes \nabla'_{\frac{ \mathbf{l-m}}{ \mathbf{r}} }$. Since these connections are compatible when $\frac{ \mathbf{m}}{ \mathbf{r}}$ varies in  $\frac{1}{\mathbf r}\mathbb{Z}^I$, Lemma \ref{lem:connection_on_colimit} shows that they give rise to a natural logarithmic connection on $\left( \mathcal{E}_\cdot \otimes \mathcal{E}'_\cdot \right)_{\frac{ \mathbf{l}}{ \mathbf{r}} }$ that we denote by $\left( \nabla_\cdot \otimes \nabla'_\cdot \right)_{\frac{ \mathbf{l}}{ \mathbf{r}} }$. The functoriality in \( \frac{ \mathbf{l}}{ \mathbf{r}}  \) is also clear, so we have lifted the tensor product from $\Par_{\frac{1}{\mathbf r}}(X,\mathbf D)$ to $\Par\Con_{\frac{1}{\mathbf r}}(X,\mathbf D)$. 

We now prove our main result, that is, the correspondence between \parabolic connections and holomorphic connections on the stack of roots. This can been seen as a de Rham version of the results for vector bundles in \cite{bor:corr,bor:rep}. However the same strategy of proof does not apply: namely holomorphic connections are not locally sum of connections of rank $1$. Our proof  is based mainly on Theorem \ref{thm:car_hol_stacky} and on the following larger equivalence of categories, which rather uses the aforementioned results.

\begin{proposition}
\label{prop:larger_equivalence}

The functors $(\mathcal E_\cdot,\nabla_\cdot)\mapsto  (\widehat{\mathcal E_\cdot},\widehat{\nabla_\cdot})$ and 
$(\mathcal F,\nabla)\mapsto (\widehat{\mathcal F}_\cdot,\widehat{\nabla}_\cdot)$ are inverse tensor equivalences of categories between $\Par\Con_{\frac{1}{\mathbf r}}(X,\mathbf D)$ and $\Con( \mathfrak{X},\mathfrak{D})$.

\end{proposition}

\begin{proof}

	According to Corollary \ref{cor:sections-logarithmic-Atiyah-extension} (resp. Corollary \ref{cor:sections-parabolic-Atiyah-extension})	the category $\Con( \mathfrak{X},\mathfrak{D})$ (resp $\Par\Con_{\frac{1}{\mathbf r}}(X,\mathbf D)$) is equivalent to the category $\Sec( \mathfrak{X},\mathfrak{D})$ (resp. $\Par\Sec_{\frac{1}{\mathbf r}}(X, \mathbf D)$). It is thus sufficient to show that the functors $\mathcal E_\cdot \mapsto \widehat{\mathcal E_\cdot}$ and $\mathcal F \mapsto \widehat{\mathcal F}_\cdot$ (Definition \ref{def:functors_F_G}) induce inverse equivalences between $\Sec( \mathfrak{X},\mathfrak{D})$ and $\Par\Sec_{\frac{1}{\mathbf r}}(X, \mathbf D)$.

Let $\mathcal F$ be a vector bundle on $\mathfrak X$. The projection formula shows that the natural morphism 
$P^1_{(X,D)/k}(\pi_*\mathcal F)  \rightarrow \pi_*P^1_{(\mathfrak{X},\mathfrak{D})/k}(\mathcal F)$ is an isomorphism.
From this and Lemma \ref{lem:twist_Atiyah_exact_sequence} it follows that the functor $\mathcal F \mapsto \widehat{\mathcal F}_\cdot$ sends the logarithmic Atiyah exact sequence of $\mathcal F$ to the parabolic Atiyah exact sequence of $\widehat{\mathcal F}_\cdot$. Hence the result follows from Theorem \ref{thm:description}.

The fact that these equivalences preserve tensor products follows from Fubini's formula for coends, see \cite[\S 3.4.4]{bor:corr}.  
\end{proof}

\begin{theorem}[]
\label{thm:correspondence}
The functors $(\mathcal E_\cdot,\nabla_\cdot)\mapsto  (\widehat{\mathcal E_\cdot},\widehat{\nabla_\cdot})$ and 
$(\mathcal F,\nabla)\mapsto (\widehat{\mathcal F}_\cdot,\widehat{\nabla}_\cdot)$ are inverse tensor equivalences of categories between $\Par\Con_{\frac{1}{\mathbf r}}^{st}(X,\mathbf D)$ and $\Con( \mathfrak{X})$. 
\end{theorem}

\begin{proof}
This follows from Proposition \ref{prop:larger_equivalence} and Theorem \ref{thm:car_hol_stacky}.
\end{proof}

\begin{remark}
\label{rem:correspondence}
Let $(\mathcal E_\cdot,\nabla_\cdot)$ be a \parabolic connection. The theorem shows that the natural connection on the vector bundle $\widehat{\mathcal E_\cdot}=\int^{\frac{1}{\mathbf r}\mathbb Z^I}\pi^*\mathcal E_\cdot \otimes \mathcal O_{\mathfrak X}( \cdot  \mathbf r \mathbf{\mathfrak D ) } $ is holomorphic. But, in most cases, the connection $\pi^* \nabla_{\frac{ \mathbf{l} }{ \mathbf{r}} }( \mathbf{l} \mathbf{\mathfrak D} )$ on an individual term $\pi^*\mathcal E_{\frac{ \mathbf{l} }{ \mathbf{r}} } \otimes \mathcal O_{\mathfrak X}( \mathbf{l} \mathbf{\mathfrak D )}$ is not holomorphic (this is already true for the simplest example of $({{\mathcal O}_X}_\cdot, d_\cdot)$). 
\end{remark}

\begin{corollary}
\label{cor:correspondence}
Let $(\mathcal F,\nabla)$ and $(\mathcal F',\nabla')$ be two holomorphic connections on $\mathfrak{X}$. Then any isomorphism  
$(\pi_*\mathcal F,\pi_*\nabla)\simeq (\pi_*\mathcal F',\pi_*\nabla')$ lifts uniquely to an isomorphism $(\mathcal F,\nabla)\simeq(\mathcal F',\nabla')$.
\end{corollary}

\begin{proof}
	This follows from Theorem \ref{thm:correspondence} and Corollary \ref{cor:parabolic-connection-semisimple-residues}.
\end{proof}

\begin{remark}
\label{rem:cor-correspondence}
The corresponding statement for logarithmic connections is false. Namely let us assume for simplicity that $\#I=1$. Let $\omega \in \Gamma(\mathfrak{X},\Omega^1_{\mathfrak{X}/k}(\log(\mathfrak{D}))$, and $0\leq l <r$. The morphism $\mathcal{O}_{\mathfrak{X}}\rightarrow \mathcal{O}_{\mathfrak{X}}(l \mathfrak{D})$ becomes an isomorphism after applying $\pi_*$  (\cite[Lemme 3.11]{bor:corr}) and is compatible with the connections $d+\omega$ (resp. $(d+\omega)(l \mathfrak{D})$) on $\mathcal{O}_{\mathfrak{X}}$ (resp. $\mathcal{O}_{\mathfrak{X}}(l \mathfrak{D})$), see Remark \ref{rem:canonical connection}. We deduce that $(\pi_*(\mathcal{O}_{\mathfrak{X}}(l \mathfrak{D})),\pi_*(d+\omega)(l \mathfrak{D}))=(\mathcal O_X,d+\omega)$ is independent of $l$, where we now see $\omega$ as an element of $\Gamma(X,\Omega^1_{X/k}(\log({D}))$.
\end{remark}

\subsection{The case of $\lambda$-connections}%
\label{sub:the_case_of_lambda_connections}

We fix $\lambda$ in the base field $k$. Let us discuss briefly what happens more generally for $\lambda$-connections, that is pairs $(\mathcal E,\nabla)$, where $\mathcal E$ is a vector bundle and $\nabla : \mathcal{E} \rightarrow  \mathcal{E}\otimes_{\mathcal{O}_X}\Omega^1(\log D)$ is $k$-linear morphism  verifying the $\lambda$-Leibniz rule $\nabla(fs)=f\nabla(s)+\lambda df\otimes s$. Such a connection can be twisted by a divisor $B$ with support in $D$ and $\res_{D_i} (\nabla(B))= \left(\res_{D_i} (\nabla)\right)(B)-\lambda\mu_i  \id$, where $\mu_i$ is the valuation of $B$ at $D_i$.

There is an obvious notion of a \preparabolic $\lambda$-connection generalizing Definition \ref{def:pre-parabolic_connection}, one just replaces connections by $\lambda$-connections. A  \preparabolic $\lambda$-connection $(\mathcal E_{\cdot},\nabla_{\cdot})$ on  $(X, \mathbf D)$ with weights in $\frac{1}{\mathbf r}\mathbb Z^I$ is \parabolic if for each $i\in I$ the morphism 
$\res_{|D_i}(\nabla_{\frac{\mathbf{l}}{\mathbf{r}}})$ induces \(\lambda \frac{l_i}{r_i} \id \) on  \(  \frac{\mathcal E_{ \frac{\mathbf{l}}{\mathbf{r}}} }{\mathcal E_{ \frac{\mathbf{l}+e_i}{\mathbf{r}}} }\).

It is clear that Theorem \ref{thm:car_hol_stacky} and Theorem \ref{thm:correspondence} hold for $\lambda$-connections, with the same proofs. The reason we have not written the article at this level of generality is that the only real new content is for $\lambda=0$, that is, parabolic Higgs bundles, a case already known from \cite{BMW:higgs}. Our choice of the term `strongly parabolic connection' is motivated by the case of parabolic Higgs bundles. 

Finally, if $(\mathcal E_{\cdot},\nabla_{\cdot})$ is a \parabolic $\lambda$-Higgs bundle, then it is clear from the definition that $\res_{|D_i}(\nabla_0)$ is nilpotent. Moreover, similarly to Proposition \ref{prop:parabolic-connection-semisimple-residues}, the Higgs field $\nabla_0$ should enable to reconstruct the parabolic structure. 

\section{Towards the log-Kummer algebraic fundamental group}%
\label{sec:towards_the_log-kummer_algebraic_fundamental_group}

In this part, we assume that $k$ is a field of characteristic $0$, and use the notations of the log-smooth context (\S \ref{ssub:logarithmic_context}).

\subsection{Curvature of a parabolic connection}%
\label{sub:curvature_of_a_pre-parabolic_connection}

The usual definition of curvature admits a straightforward transposition to the parabolic context:

\begin{definition}[]
\label{def:curvature_of_a_pre-parabolic_connection}

Let $(\mathcal{E}_\cdot, \nabla_\cdot)\in  \Par\Con_{\frac{1}{\mathbf r}}(X,\mathbf D)$ be a \preparabolic connection on $(X,\mathbf{D})$. Its curvature $C_{(\mathcal{E}_\cdot, \nabla_\cdot)}$ is defined as the composite morphism:
\[C_{(\mathcal{E}_\cdot, \nabla_\cdot)} : \mathcal{E}_\cdot \xrightarrow{\nabla} \Omega_{X/k}^1(\log(D)) \otimes_{ \mathcal{O} _X} \mathcal{E}_\cdot \xrightarrow{\id\otimes \nabla} \Omega_{X/k}^2(\log(D)) \otimes_{ \mathcal{O} _X} \mathcal{E}_\cdot \; . \] 
The \preparabolic connection is integrable if  \( C_{(\mathcal{E}_\cdot, \nabla_\cdot)}=0 \). \end{definition}

\begin{remark}
\label{rem:curvature_of_a_pre-parabolic_connection}

\begin{enumerate}
	\item As usual $C_{(\mathcal{E}_\cdot, \nabla_\cdot)}$ is in fact $\mathcal{O}_X$-linear.
	\item By definition $(\mathcal{E}_\cdot, \nabla_\cdot)$ is integrable if it is componentwise.
	\item If $(\mathcal{F},\nabla)$ is the logarithmic connection on the stack of roots  $\sqrt[\mathbf{r}]{{(\mathbf{D},\mathbf{s})/X}}$ associated to $(\mathcal{E}_\cdot, \nabla_\cdot)$ (see Proposition \ref{prop:larger_equivalence}) then $C_{(\mathcal{E}_\cdot, \nabla_\cdot)}$ corresponds to the curvature $C_{(\mathcal{F},\nabla)}$ via the correspondence of Theorem \ref{thm:description}. 
\end{enumerate}
\end{remark}

\subsection{Algebraic fundamental groups of Deligne-Mumford stacks}%
\label{sub:algebraic_fundamental_groups_of_deligne_mumford_stacks}

\begin{proposition}
\label{prop:int_connections_DM_stack_tannaka}
Let $\mathfrak{X}/k$ be a smooth Deligne-Mumford stack. The category $\Int\Con( \mathfrak{X})$ of integrable holomorphic connections on  $\mathfrak{X}/k$ is tannakian. 
\end{proposition}

\begin{proof}
	The usual proof for schemes applies to Deligne-Mumford stacks as well: see for instance \cite[VI 1.2]{saa:tan}.
\end{proof}

\begin{corollary}
\label{cor:algebraic_fundamental_groups_of_deligne_mumford_stacks}
The category $\Int\Par\Con_{\frac{1}{\mathbf r}}^{st}(X,\mathbf D)$ of integrable \parabolic connections with weights in $\frac{1}{\mathbf r}\mathbb Z^I$ is tannakian.
\end{corollary}

\begin{proof}
	According to Theorem \ref{thm:correspondence} and Remark \ref{rem:curvature_of_a_pre-parabolic_connection}, the category $\Int\Par\Con_{\frac{1}{\mathbf r}}(X,\mathbf D)$ is equivalent as a tensor category to the category $\Int\Con(\sqrt[\mathbf{r}]{{(\mathbf{D},\mathbf{s})/X}})$, hence the result follows from Proposition \ref{prop:sncd_implies_smooth} and Proposition \ref{prop:int_connections_DM_stack_tannaka}.
\end{proof}

\begin{proposition}
\label{prop:quotients_of_stacky_alg_fundamental_group}
Let $x\in X(k)\backslash D(k)$ and $\pi_{\mathbf D, \mathbf{r}}^{alg}(X,x)$ be the fundamental group of the Tannaka category $\Int\Par\Con_{\frac{1}{\mathbf r}}^{st}(X,\mathbf D)$ based at $x$. Fix an affine algebraic group $G/k$. Then there is a one to one correspondence between:
\begin{itemize}
	\item morphisms $\pi_{\mathbf D, \mathbf{r}}^{alg}(X,x) \rightarrow G$,
	\item triples $(T\rightarrow \sqrt[\mathbf{r}]{{(\mathbf{D},\mathbf{s})/X}}, t, \nabla)$ where $T\rightarrow \sqrt[\mathbf{r}]{{(\mathbf{D},\mathbf{s})/X}}$ is a $G$-torsor, $t\in T(k)$ is a lifting of $x$, and $\nabla$ is an integrable connection on $T\rightarrow \sqrt[\mathbf{r}]{{(\mathbf{D},\mathbf{s})/X}}$.
\end{itemize}

\end{proposition}

\begin{proof}
	It is classical that $G$-torsors over $\sqrt[\mathbf{r}]{{(\mathbf{D},\mathbf{s})/X}}$ endowed with an integrable connection correspond to tensor functors $\Rep_k(G) \rightarrow \Int\Con(\sqrt[\mathbf{r}]{{(\mathbf{D},\mathbf{s})/X}})$, so the result follows from Tannaka duality, Theorem \ref{thm:correspondence} and Remark \ref{rem:curvature_of_a_pre-parabolic_connection}.
\end{proof}

\subsection{A candidate for the log-Kummer algebraic fundamental group}%
\label{sub:a_candidate_for_the_log_kummer_algebraic_fundamental_group}

Let $(X,(D_i)_{i\in I})$ be a log-scheme, in the restricted set-up described \S \ref{ssub:logarithmic_differentials}.
Let $x\in X(k)\backslash D(k)$. We can define:

\[ \pi_{\mathbf D}^{alg}(X,x) =\varprojlim_{\mathbf r} \pi_{\mathbf D, \mathbf{r}}^{alg}(X,x) \; . \] 
This is the Tannaka group of the category

\[\Int\Par\Con^{st}(X,\mathbf D) =\varinjlim_{\mathbf r} \Int\Par\Con_{\frac{1}{\mathbf r}}^{st}(X,\mathbf D) \; .\] 
Assume now that $G$ is an algebraic group of finite type over $k$. Then Proposition \ref{prop:quotients_of_stacky_alg_fundamental_group} suggests that there is a one to one correspondence between:
\begin{itemize}
	\item morphisms $\pi_{\mathbf D}^{alg}(X,x) \rightarrow G$,
	\item triples $(T\rightarrow \sqrt[\mathbf{\infty}]{{(\mathbf{D},\mathbf{s})/X}}, t, \nabla)$ where $T\rightarrow \sqrt[\mathbf{\infty}]{{(\mathbf{D},\mathbf{s})/X}}$ is a $G$-torsor, $t\in T(k)$ is a lifting of $x$, and $\nabla$ is an integrable connection on $T\rightarrow \sqrt[\mathbf{\infty}]{{(\mathbf{D},\mathbf{s})/X}}$.
\end{itemize}
Here \[ \sqrt[\mathbf{\infty}]{{(\mathbf{D},\mathbf{s})/X}}=\varprojlim_{\mathbf r} \sqrt[\mathbf{r}]{{(\mathbf{D},\mathbf{s})/X}} \]
is a pro-algebraic stack called the infinite root stack in \cite[]{TV:infinite}. This works hints on the other hand that $G$-torsors on  $\sqrt[\mathbf{\infty}]{{(\mathbf{D},\mathbf{s})/X}}$ for the étale topology correspond to $G$-torsors on the log scheme $(X,(D_i)_{i\in I})$ for the Kummer-étale topology. So we think that the group $\pi_{\mathbf D}^{alg}(X,x)$ might deserve the name of log-Kummer algebraic fundamental group of $(X,(D_i)_{i\in I})$.
\printbibliography
\end{document}